\documentclass[a4paper,12pt,pdfmx]{article}
\usepackage{amsfonts}
\usepackage{amssymb}
\usepackage{amsmath}
\usepackage{amsthm}
\usepackage{bm}
\usepackage{here}
\usepackage{graphicx}
\usepackage{tikz}
\usetikzlibrary{intersections,calc,arrows}
\usepackage{color}
\usepackage{cases}
\numberwithin{equation}{section}

\theoremstyle{definition}
\newtheorem{defn}{Definition}[section]
\newtheorem{rem}{Remark}[section]
\theoremstyle{plain}
\newtheorem{lemm}{Lemma}[section]
\newtheorem{prop}{Proposition}[section]
\newtheorem{thm}{Theorem}[section]

\newtheorem{mainthm}{Main Theorem}[section]
\title{Hypergeometric solutions for variants of the $q$-hypergeometric equation}
\author{
	Taikei Fujii
	\ \ and  
	Takahiko Nobukawa
	\\
}
\date{}
\allowdisplaybreaks
\begin{document}
\maketitle
\renewcommand{\labelenumi}{\rm (\arabic{enumi})}
\begin{abstract}
	We introduce a configuration of a $q$-difference equation and characterize the variants of the $q$-hypergeometric equation, which were defined by Hatano-Matsunawa-Sato-Takemura, by configurations.
	We show integral solutions and series solutions for the variants of the $q$-hypergeometric equation.
	
	Key words: $q$-hypergeometric equation, Jackson integral of Jordan-Pochhammer type, very-well-poised $q$-hypergeometric equation, point configuration.
	
	2020 Mathematics Subject Classification Numbers: 33D60, 33D15, 39A13.
\end{abstract}

\section{Introduction}\label{secintro}

The Gauss hypergeometric equation
\begin{align}\label{Gausseq}
	\left[x(1-x)\frac{d^{2}}{dx^{2}}+(\gamma-(\alpha+\beta+1)x)\frac{d}{dx}-\alpha\beta\right]f(x)=0,
\end{align}
is a standard form of second order Fuchsian differential equations with three singularities{ on $\mathbb{P}^{1}=\mathbb{C}\cup\{\infty\}$}.
The equation (\ref{Gausseq}) is characterized by the following Riemann scheme:
\begin{align}\label{GaussRiesch}
	\left\{\begin{array}{ccc}
		x=0&x=1&x=\infty\\
		0&0&\alpha\\
		1-\gamma&\gamma-\alpha-\beta&\beta
	\end{array}\right\}.
\end{align}
The equation (\ref{Gausseq}) has the integral solution
\begin{align}\label{Gaussint}
	\int_{C}t^{\alpha-1}(1-t)^{\gamma-\alpha-1}(1-xt)^{-\beta}dt,
\end{align}
where $C$ is a suitable path.
Also, the equation (\ref{Gausseq}) has the series solution
\begin{align}\label{Gaussser}
	{}_{2}F_{1}\left(\begin{array}{c}
		\alpha,\beta\\
		\gamma
	\end{array};x\right)=1+\frac{\alpha\cdot\beta}{\gamma\cdot1}x+\frac{\alpha(\alpha+1)\cdot\beta(\beta+1)}{\gamma(\gamma+1)\cdot1\cdot2}x^{2}+\cdots.
\end{align}
The Riemann scheme (\ref{GaussRiesch}) and the integral solution (\ref{Gaussint}), the series solution (\ref{Gaussser}) are basic properties of the Gauss equation (\ref{Gausseq}).

On the other hand, the Riemann-Papperitz differential equation
\begin{align}\label{RiePapeq}
	\notag&\Biggl[\frac{d^{2}}{dz^{2}}+\left(\frac{1-\alpha_{1}-\beta_{1}}{z-t_{1}}+\frac{1-\alpha_{2}-\beta_{2}}{z-t_{2}}+\frac{1-\alpha_{3}-\beta_{3}}{z-t_{3}}\right)\frac{d}{dz}+\frac{1}{(z-t_{1})(z-t_{2})(z-t_{3})}\\
	&\times\left(\frac{\alpha_{1}\beta_{1}(t_{1}-t_{2})(t_{1}-t_{3})}{z-t_{1}}+\frac{\alpha_{2}\beta_{2}(t_{2}-t_{1})(t_{2}-t_{3})}{z-t_{2}}+\frac{\alpha_{3}\beta_{3}(t_{3}-t_{1})(t_{3}-t_{2})}{z-t_{3}}\right)\Biggr]g(z)=0,
\end{align}
where $\alpha_{1}+\alpha_{2}+\alpha_{3}+\beta_{1}+\beta_{2}+\beta_{3}=1$, has also similar properties as for the Gauss equation (\ref{Gausseq}).
The equation (\ref{RiePapeq}) is characterized by the Riemann scheme
\begin{align}\label{RiePapsch}
	\left\{\begin{array}{ccc}
		z=t_{1}&z=t_{2}&z=t_{3}\\
		\alpha_{1}&\alpha_{2}&\alpha_{3}\\
		\beta_{1}&\beta_{2}&\beta_{3}
	\end{array}\right\}.
\end{align}
Thus we put $\displaystyle x=\frac{z-t_{1}}{z-t_{3}}\frac{t_{2}-t_{3}}{t_{2}-t_{1}}$ and $h(x)=x^{-\alpha_{1}}(1-x)^{-\alpha_{2}}g(z)$, then the function $h(x)$ satisfies the Gauss equation (\ref{Gausseq}) with $\alpha=\alpha_{1}+\alpha_{2}+\alpha_{3}$, $\beta=\alpha_{1}+\alpha_{2}+\beta_{3}$ and $\gamma=\alpha_{1}-\beta_{1}+1$.
Since the Riemann-Papperitz equation (\ref{RiePapeq}) and the Gauss equation (\ref{Gausseq}) are transformed to each other by gauge transformations, the Riemann-Papperitz equation (\ref{RiePapeq}) has the integral solution and the series solution
\begin{align}
	&\notag (z-t_{1})^{\alpha_{1}}(z-t_{2})^{\alpha_{2}}(z-t_{3})^{\alpha_{3}}\\
	&\times \label{RiePapint}\int_{C}(t-t_{1})^{\alpha_{2}+\alpha_{3}+\beta_{1}-1}(t-t_{2})^{\alpha_{1}+\alpha_{3}+\beta_{2}-1}(t-t_{1})^{\alpha_{1}+\alpha_{2}+\beta_{3}-1}(t-z)^{-\alpha_{1}-\alpha_{2}-\alpha_{3}}dt,\\
	&\label{RiePapser}\left(\frac{z-t_{1}}{z-t_{3}}\right)^{\alpha_{1}}\left(\frac{z-t_{2}}{z-t_{3}}\right)^{\alpha_{2}}{}_{2}F_{1}\left(\begin{array}{c}
		\alpha_{1}+\alpha_{2}+\alpha_{3}, \alpha_{1}+\alpha_{2}+\beta_{3}\\
		\alpha_{1}-\beta_{1}+1
	\end{array};\frac{z-t_{1}}{z-t_{3}}\frac{t_{2}-t_{3}}{t_{2}-t_{1}}\right).
\end{align}

A $q$-difference analog of the Gauss hypergeometric equation (\ref{Gausseq}) was introduced as follows:
\begin{align}\label{Heineeq}
	[x(1-aT_{x})(1-bT_{x})-(1-T_{x})(1-cq^{-1}T_{x})]f(x)=0,
\end{align}
where $T_{x}$ is the $q$-shift operator {of $x$}, that is, $T_{x}f(x)=f(qx)$.
This equation is called Heine's $q$-hypergeometric equation.
In this paper we fix $q\in\mathbb{C}$ with $0<|q|<1$.
The equation (\ref{Heineeq}) has the integral solution and the series solution
\begin{align}
	\label{Heineint}&\int_{C}t^{\alpha-1}\frac{(qt)_{\infty}}{(ct/a)_{\infty}}\frac{(bxt)_{\infty}}{(xt)_{\infty}}d_{q}t,\\
	\label{Heineser}&{}_{2}\varphi_{1}\left(\begin{array}{c}
		a,b\\
		c
	\end{array}x\right),
\end{align}
where $a=q^{\alpha}$, and
\begin{align}
	&\int_{0}^{\tau}f(t)d_{q}t=(1-q)\sum_{n=0}^{\infty}f(\tau q^{n})\tau q^{n},\ \ \int_{\tau_{1}}^{\tau_{2}}f(t)d_{q}t=\int_{0}^{\tau_{2}}f(t)d_{q}t-\int_{0}^{\tau_{1}}f(t)d_{q}t,\\
	&(a)_{\infty}=\prod_{n=0}^{\infty}(1-aq^{n}),\ \ (a)_{m}=\frac{(a)_{\infty}}{(aq^{m})_{\infty}},\ \ 
	(a_{1},\ldots,a_{r})_{m}=(a_{1})_{m}\cdots(a_{r})_{m},\\
	&{}_{r}\varphi_{s}\left(\begin{array}{c}
		a_{1},\ldots,a_{r}\\
		b_{1},\ldots,b_{s}
	\end{array};x\right)=\sum_{n=0}^{\infty}\frac{(a_{1},\ldots,a_{r})_{n}}{(b_{1},\ldots,b_{s},q)_{n}}(-1)^{(r-s-1)n}q^{(r-s-1)\binom{n}{2}}x^{n},
\end{align}
are called the Jackson integral of $f(t)$, the $q$-Pochhammer symbol and the $q$-hypergeometric series, respectively.
The integral (\ref{Heineint}) and the series (\ref{Heineser}) are $q$-analogs of (\ref{Gaussint}) and (\ref{Gaussser}), respectively.
In the theory of $q$-difference equations, the points $x=0$ and $x=\infty$ are the fixed points of the $q$-shift operator $T_{x}$.
Thus only the points $x=0$ and $x=\infty$ can be singularities for $q$-difference equations.
The Gauss equation (\ref{Gausseq}) has singularities at $x=0$ and $x=\infty$, and then it is easy to obtain a $q$-analog of (\ref{Gausseq}).
However, the points $z=0$ and $z=\infty$ are regular points for the Riemann-Papperitz equation (\ref{RiePapeq}).
Thus it is difficult to consider properties for a $q$-analog of (\ref{RiePapeq}).

In \cite{HMST}, the variants of the $q$-hypergeometric equation of degree two and degree three were introduced.
The variant of the $q$-hypergeometric equation of degree two, {i.e. $\deg_{x}\mathcal{H}_{2}=2$,} is defined as a special case of the $q$-Heun equation, as follows:
\begin{align}
	\mathcal{H}_{2}&f(x)=0,\\
	\mathcal{H}_{2}&=\prod_{i=1}^{2}(x-q^{h_{i}+1/2}t_{i})\cdot T_{x}^{-1}+q^{\alpha_{1}+\alpha_{2}}\prod_{i=1}^{2}(x-q^{l_{i}-1/2}t_{i})\cdot T_{x}-(q^{\alpha_{1}}+q^{\alpha_{2}})x^{2}\notag\\
	&+Ex+p(q^{1/2}+q^{-1/2})t_{1}t_{2},\\
	p&=q^{(h_{1}+h_{2}+l_{1}+l_{2}+\alpha_{1}+\alpha_{2})/2},\ \ E=-p\{(q^{-h_{2}}+q^{-l_{2}})t_{1}+(q^{-h_{1}}+q^{-l_{1}})t_{2}\}.
\end{align}
The $q$-Heun equation was introduced in \cite{Hahn1971}.
Heun's differential equation is a second order Fuchsian differential equation with four singularities.
When one of four singularities is essentially non-singular, Heun's differential equation can be transformed to the Gauss hypergeometric equation with some gauge transformation.
Here, we say a singularity of a second order differential equation is essentially non-singular when the difference between the characteristic exponents is 1 and the singularity is non-logarithmic.
The way to derive the equation $\mathcal{H}_{2}f(x)=0$ from the $q$-Heun equation is also to specialize so that the point $x=0$ is essentially non-singular.
The equation $\mathcal{H}_{2}f(x)=0$ is a $q$-analog of a second order Fuchsian differential equation with four singularities $\{0,t_{1},t_{2},\infty\}$, and the point $x=0$ is essentially non-singular.
Thus this equation is essentially a $q$-analog of the Riemann-Papperitz equation with $t_{3}=\infty$.
Note that by taking the limit $t_{2}\to0$, the variant of the $q$-hypergeometric equation of degree two becomes Heine's $q$-hypergeometric equation with some change of parameters.
Several solutions of the equation $\mathcal{H}_{2}f(x)=0$ were shown in \cite{HMST} as follows:
\begin{align}
	\label{introH2int}&x^{-\alpha_{1}}\Phi^{(1)}\left(\begin{array}{c}
		q^{\lambda_{0}+\alpha_{1}}; q^{\lambda_{0}+\alpha_{1}-h_{2}+l_{2}}, q^{\lambda_{0}+\alpha_{1}-h_{1}+l_{1}}\\
		q^{\alpha_{1}-\alpha_{2}+1}
	\end{array};q^{l_{1}+1/2}\frac{t_{1}}{x},q^{l_{2}+1/2}\frac{t_{2}}{x}\right),\\
	\label{introH2ser}&x^{\lambda_{0}}{}_{3}\varphi_{2}\left(\begin{array}{c}
		x/(q^{l_{1}-1/2}t_{1}), q^{\lambda_{0}+\alpha_{1}}, q^{\lambda_{0}+\alpha_{2}}\\
		q^{h_{1}-l_{1}+1},q^{h_{2}-l_{2}+1}t_{2}/t_{1}
	\end{array};q\right).
\end{align}
Here, $\lambda_{0}=(h_{1}+h_{2}-l_{1}-l_{2}-\alpha_{1}-\alpha_{2}+1)/2$ and $\Phi^{(1)}$ is the $q$-Appell hypergeometric series
\begin{align}
	\Phi^{(1)}\left(\begin{array}{c}
		a; b_{1}, b_{2}\\
		c
	\end{array};x_{1},x_{2}\right)=\sum_{n_{1}=0}^{\infty}\sum_{n_{2}=0}^{\infty}\frac{(a)_{n_{1}+n_{2}}(b_{1})_{n_{1}}(b_{2})_{n_{2}}}{(c)_{n_{1}+n_{2}}(q)_{n_{1}}(q)_{n_{2}}}x_{1}^{n_{1}}x_{2}^{n_{2}}.
\end{align}
The $q$-Appell series $\Phi^{(1)}$ has a Jackson integral representation \cite{And}
\begin{align}
	\Phi^{(1)}\left(\begin{array}{c}
		a; b_{1}, b_{2}\\
		c
	\end{array};x_{1},x_{2}\right)=\frac{(a,c/a)_{\infty}}{(q,c)_{\infty}(1-q)}\int_{0}^{1}t^{\alpha-1}\frac{(qt,b_{1}x_{1}t,b_{2}x_{2}t)_{\infty}}{(ct/a,x_{1}t,x_{2}t)_{\infty}}d_{q}t,
\end{align}
where $a=q^{\alpha}$.
Therefore the solutions (\ref{introH2int}) and (\ref{introH2ser}) are $q$-analogs of the functions (\ref{RiePapint}) and (\ref{RiePapser}) with $t_{3}=\infty$, respectively.
In addition, the $q$-Heun equation was rediscovered in \cite{Takemura2017} as an eigenvalue problem for the fourth degenerated Ruijsenaars-van Diejen operator \cite{Ru,vD} of one variable.
In \cite{Takemura2018}, the variants of the $q$-Heun equation of degree three and degree four were introduced from the viewpoint of degenerations of the Ruijsenaars-van Diejen operator.
The variant of the $q$-hypergeometric equation of degree three is also defined by specializing the variant of the $q$-Heun equation of degree three so that $x=0$ is essentially non-singular, as follows:
\begin{align}
	\mathcal{H}_{3}&f(x)=0,\\
	\notag \mathcal{H}_{3}&=\prod_{i=1}^{3}(x-q^{h_{i}+1/2}t_{i})\cdot T_{x}^{-1}+q^{2\alpha+1}\prod_{i=1}^{3}(x-q^{l_{i}-1/2}t_{i})\cdot T_{x}\\
	\notag &+q^{\alpha}\biggl[-(q+1)x^{3}+q^{1/2}\sum_{i=1}^{3}(q^{h_{i}}+q^{l_{i}})t_{i}x^{2}\\
	\notag &-q^{(h_{1}+h_{2}+h_{3}+l_{1}+l_{2}+l_{3}+1)/2}t_{1}t_{2}t_{3}\sum_{i=1}^{3}((q^{-h_{i}}+q^{-l_{i}})/t_{i})x\\
	&+q^{(h_{1}+h_{2}+h_{3}+l_{1}+l_{2}+l_{3})/2}(q+1)t_{1}t_{2}t_{3}\biggr].
\end{align}
{Note that} the points $x=0$ and $x=\infty$ are essentially non-singular {for the variant of the $q$-Heun equation of degree four}, and then we cannot obtain some equation by a similar specialization for the variant of the $q$-Heun equation of degree four.
The variant of the $q$-hypergeometric equation of degree three is a $q$-analog of a second order Fuchsian differential equation with five singularities $\{0,t_{1},t_{2},t_{3},\infty\}$, and $x=0$ and $x=\infty$ are essentially non-singular.
Therefore the variant of the $q$-hypergoemetric equation of degree three is essentially a $q$-analog of the Riemann-Papperitz equation.
Note that the equation $\mathcal{H}_{3}f(x)=0$ becomes $\mathcal{H}_{2}f(x)=0$ in the limit $t_{3}\to\infty$ {(see \cite{HMST} or Remark \ref{Remdegeeq})}.
However, the explicit solution for the variant of the $q$-hypergeometric equation of degree three was not given in \cite{HMST}.
We have seen only conjectural solutions in \cite{HMST}.

Our aim is to obtain the integral and series solutions for the variants of the $q$-hypergeometric equation, which are $q$-analogs of the integral (\ref{RiePapint}) and the series (\ref{RiePapser}), respectively.
{We will introduce a configuration of linear $q$-difference equations in order to characterize $q$-difference equations in Definition \ref{defconfi} below.
In this paper, we will consider the following equation:
\begin{align}
	\mathcal{E}_{3}&f(x)=0,\\
	\mathcal{E}_{3}&=[x^{3}(B-AT_{x})(B-AqT_{x})-x^{2}(e_{1}(a)-qe_{1}(b)T_{x})(B-AT_{x})\notag\\
	&+x(e_{2}(a)-qe_{2}(b)T_{x})(1-T_{x})-e_{3}(a)B^{-1}(1-q^{-1}T_{x})(1-T_{x})]T_{x}^{-1},
\end{align}
where $a=(a_{1},a_{2},a_{3})$, $b=(b_{1},b_{2},b_{3})$, $a_{1}a_{2}a_{3}A=q^{2}b_{1}b_{2}b_{3}B$ and $e_{i}$ is the elementary symmetric polynomial of degree $i$.
The configuration of the equation $\mathcal{E}_{3}f(x)=0$ is in the same form as the configuration of the equation $\mathcal{H}_{3}f(x)=0$.
Thus by some gauge transformation and changes of parameters, the equation $\mathcal{E}_{3}f(x)=0$ is identified with the equation $\mathcal{H}_{3}f(x)=0$ (see Figure \ref{conqhyp3}, Figure \ref{conE3} and Remark \ref{remE3toH3}).
Also, in this paper we will consider the following equation $\mathcal{E}_{2}f(x)=0$ instead of $\mathcal{H}_{2}f(x)=0$:
\begin{align}
	\mathcal{E}_{2}&f(x)=0,\\
	\mathcal{E}_{2}&=[x^{2}(1-q^{\alpha}T_{x})(B-AT_{x})-x(e_{1}(a)-qe_{1}(b)T_{x})(1-T_{x})\notag\\
	&-e_{2}(a)B^{-1}(1-q^{-1}T_{x})(1-T_{x})]T_{x}^{-1},
\end{align}
where $a=(a_{1},a_{2})$, $b=(b_{1},b_{2})$ and $a_{1}a_{2}A=q^{\alpha+1}b_{1}b_{2}B$.
For the same reason as for the identification of the equation $\mathcal{E}_{3}f(x)=0$ and $\mathcal{H}_{3}f(x)=0$, the equation $\mathcal{E}_{2}f(x)=0$ is identified with $\mathcal{H}_{2}f(x)=0$ (see Figure \ref{conqhyp2}, Figure \ref{conE2} and Remark \ref{remE2toH2}).}
We focus on the solution (\ref{introH2int}) for the variant of the $q$-hypergeometric equation of degree two.
As mentioned above, the solution (\ref{introH2int}) has the Jackson integral representation.
This integral is a special case of the Jackson integral of Jordan-Pochhammer type defined by
\begin{align}\label{introJPint}
	\int_{C}t^{\alpha-1}\prod_{i=0}^{M}\frac{(a_{i}t)_{\infty}}{(b_{i}t)_{\infty}}d_{q}t.
\end{align}
Note that the integral (\ref{Heineint}) is the case $M=1$ of (\ref{introJPint}).
Since a solution for the equation $\mathcal{H}_{2}f(x)=0$ is written by the case $M=2$ of the integral (\ref{introJPint}), we expect that there is a solution for the equation $\mathcal{H}_{3}f(x)=0$ written by the integral (\ref{introJPint}) {with $M=3$}.
One of the main results is Theorem \ref{thmint3}, which gives integral solutions for the variant of the $q$-hypergeometric equation of degree three.
\begin{mainthm}[Theorem \ref{thmint3}]
	Let $\tau_{1}$, $\tau_{2}\in\{q/a_{1},q/a_{2},q/a_{3},q/(Ax)\}$ and $\sigma_{1}$, $\sigma_{2}\in\{b_{1},b_{2},b_{3},Bx\}$.
	We suppose $q^{\lambda}=B/A\notin q^{\mathbb{Z}}$. 
	Then the integrals
	\begin{align}
		\label{mainint1}&\int_{\tau_{1}}^{\tau_{2}}\frac{(Axt,a_{1}t,a_{2}t,a_{3}t)_{\infty}}{(Bxt,b_{1}t,b_{2}t,b_{3}t)_{\infty}}d_{q}t,\\
		\label{mainint2}&x^{\lambda}\int_{\sigma_{1}}^{\sigma_{2}}\frac{(qs/(Bx),qs/b_{1},qs/b_{2},qs/b_{3})_{\infty}}{(qs/(Ax),qs/a_{1},qs/a_{2},qs/a_{3})_{\infty}}d_{q}s,
	\end{align}
	satisfy the equation $\mathcal{E}_{3}f(x)=0$.
\end{mainthm}
These integrals are $q$-analogs of the integral (\ref{RiePapint}).
The integral solutions for the variant of the $q$-hypergeometric equation of degree two can be obtained by taking some limit to the integrals (\ref{mainint1}), (\ref{mainint2}).
For the integral solutions for the variant of the $q$-hypergeometric equation of degree two, see Theorem \ref{thmint2}.
Also another main result is Theorem \ref{thmser3}, which gives series solutions for  the variant of the $q$-hypergeometric equation of degree three.
\begin{mainthm}[Theorem \ref{thmser3}]\label{Mainthm2}
	The functions
	\begin{align}
		&\frac{(Axq/a_{2})_{\infty}}{(Bxq/a_{2})_{\infty}}{}_{8}W_{7}\left(\frac{a_{3}A}{a_{2}B}; \frac{qb_{1}}{a_{2}}, \frac{qb_{2}}{a_{2}}, \frac{qb_{3}}{a_{2}}, \frac{a_{3}}{Bx}, \frac{A}{B}; \frac{qBx}{a_{1}}\right),\\
		&\frac{(a_{3}Ax/(b_{1}b_{3}),a_{3}Ax/(b_{2}b_{3}),qAx/a_{2})_{\infty}}{(qBx/a_{1},qBx/a_{2},qa_{3}Ax/(a_{2}a_{3}))_{\infty}}\notag\\
		&\times{}_{8}W_{7}\left(\frac{a_{3}Ax}{a_{2}b_{3}};\frac{qBx}{a_{2}},\frac{qb_{1}}{a_{2}},\frac{qb_{2}}{a_{2}},\frac{a_{3}}{b_{3}},\frac{Ax}{b_{3}};\frac{qb_{3}}{a_{1}}\right),\\
		&\frac{1}{x}\frac{(qa_{1}/(Ax),Ax/a_{1},a_{2}a_{3}/(b_{3}Bx),qa_{2}/(Ax),qa_{3}/(Ax))_{\infty}}{(qBx/a_{1},qb_{1}/(Ax),qb_{2}/(Ax),qb_{3}/(Ax),qa_{2}a_{3}/(b_{3}Ax))_{\infty}}\notag\\
		&\times{}_{8}W_{7}\left(\frac{a_{2}a_{3}}{b_{3}Ax};\frac{qb_{1}}{Ax},\frac{qB}{A},\frac{qb_{2}}{Ax},\frac{a_{2}}{b_{3}},\frac{a_{3}}{b_{3}};\frac{qb_{3}}{a_{1}}\right),\\
		\notag &\frac{1}{x}\frac{(qa_{1}/(Ax),qa_{2}/(Ax),qa_{3}/(Ax),Ax/a_{1},a_{2}a_{3}/(b_{1}Bx),a_{2}a_{3}/(b_{2}Bx),a_{2}a_{3}/(b_{3}Bx))_{\infty}}{(qb_{1}/(Ax),qb_{2}/(Ax),qb_{3}/(Ax),qa_{2}a_{3}/(ABx^{2}))_{\infty}}\\
		&\times{}_{8}W_{7}\left(\frac{a_{2}a_{3}}{ABx^{2}};\frac{qb_{1}}{Ax},\frac{qb_{2}}{Ax},\frac{qb_{3}}{Ax},\frac{a_{2}}{Bx},\frac{a_{3}}{Bx};\frac{qBx}{a_{1}}\right),\\
		&\frac{(qAx/a_{1},a_{1}/(Ax),a_{2}a_{3}/(b_{3}Bx))_{\infty}}{(qb_{1}/(Ax),qb_{2}/(Ax),qBx/a_{1})_{\infty}}{}_{8}W_{7}\left(\frac{a_{2}a_{3}}{a_{1}b_{3}};\frac{qBx}{a_{1}},\frac{qb_{1}}{a_{1}},\frac{qb_{2}}{a_{1}},\frac{a_{2}}{b_{3}},\frac{a_{3}}{b_{3}};\frac{qb_{3}}{Ax}\right),\\
		&\frac{(qAx/a_{1},a_{1}/(Ax),a_{2}a_{3}/(b_{1}Bx),a_{2}a_{3}/(b_{2}Bx))_{\infty}}{(qb_{1}/(Ax),qb_{2}/(Ax),qb_{3}/(Ax),qBx/a_{1},qa_{2}a_{3}/(a_{1}Bx))_{\infty}}\notag\\
		&\times{}_{8}W_{7}\left(\frac{a_{2}a_{3}}{a_{1}Bx};\frac{qb_{1}}{a_{1}},\frac{qb_{2}}{a_{1}},\frac{qb_{3}}{a_{1}},\frac{a_{2}}{Bx},\frac{a_{3}}{Bx};\frac{qB}{A}\right),
	\end{align}
	satisfy the equation $\mathcal{E}_{3}f(x)=0$.
\end{mainthm}
Here, the function ${}_{8}W_{7}(a;b,c,d,e,f;x)$ is the very-well-poised $q$-hypergeometric series
\begin{align}
	{}_{8}W_{7}(a;b,c,d,e,f;x)=\sum_{n=0}^{\infty}\frac{1-aq^{2n}}{1-a}\frac{(a,b,c,d,e,f)_{n}}{(q,qa/b,qa/c,qa/d,qa/e,qa/f)_{n}}x^{n}.
\end{align}
Some series in Main Theorem \ref{Mainthm2} formally can be considered as $q$-analogs of the series (\ref{RiePapser}).
In addition, the equation $\mathcal{H}_{3}f(x)=0$ has some symmetries, so we can get many solutions for $\mathcal{H}_{3}f(x)=0$ by some transformations acting on solutions in Main Theorem \ref{Mainthm2}.
For more details, see section \ref{secser}.


The contents of this paper are as follows.
In section \ref{seceq}, we introduce a configuration of a $q$-difference equation, and characterize the variants of the $q$-hypergeometric equation by the configurations.
In section \ref{secint}, we show the integral solutions for the variants of the $q$-hypergeometric equation.
In section \ref{secser}, we show the series solutions for the variants of the $q$-hypergeometric equation $\mathcal{E}_3f(x)=0$.
In section \ref{secsum}, we give a summary of this paper and discuss related problems.

\section{Variants of the $q$-hypergeometric equation}\label{seceq}
In this section, first we recall some fundamental concepts for linear $q$-difference equations, and define a configuration of a $q$-difference equation.
Next, we introduce variants of the $q$-hypergeometric equation defined in \cite{HMST}, and characterize them by configurations.
Variants of the $q$-hypergeometric equation are a specialization of the $q$-Heun equation or the variant of it.
Thus we introduce and characterize the $q$-Heun equation and the variant.
\subsection{Linear $q$-difference equation}\label{subseceq}
In this subsection, we recall some fundamental concepts of linear $q$-difference equations, such as characteristic roots and a non-logarithmic singularity. 
We then introduce a configuration of a $q$-difference equation.
We consider the $q$-difference equation
\begin{align}\label{geneeq}
	\mathcal{L}f(x)=0,\ \ \mathcal{L}=x^{M'}\cdot\sum_{i=0}^{M}\sum_{j=0}^{N}a_{i,j}x^{i}T_{x}^{j}\cdot (T_{x})^{N'}.
\end{align}
Also, we consider the solution
\begin{align}
	f(x)=x^{\lambda}\sum_{n=0}^{\infty}c_{n}x^{n},\ \ c_{0}=1.
\end{align}
The operator $\mathcal{L}$ can be rewritten as
\begin{align}\label{geneeqL}
	\mathcal{L}=x^{M'}(x^{M}L_{M}(T_{x})+\cdots+x^{0}L_{0}(T_{x})),
\end{align}
where $L_{i}$ is some Laurent polynomial.
Hence we have
\begin{align}
	\mathcal{L}f(x)=\sum_{n=0}^{\infty}\sum_{m=0}^{M}L_{m}(q^{\lambda+n})c_{n}x^{\lambda+M'+n+m}.
\end{align}
The condition of $\mathcal{L}f(x)=0$ is equivalent to the following conditions:
\begin{align}
	\begin{cases}
		c_{0}L_{0}(q^{\lambda})=0.\\
		c_{1}L_{0}(q^{\lambda+1})+c_{0}L_{1}(q^{\lambda})=0.\\
		\vdots
	\end{cases}
\end{align}
Since $c_{0}=1$, we have $L_{0}(q^{\lambda})=0$.
\begin{defn}
	For the equation (\ref{geneeq}), the roots of $L_{0}(y)=0$ are called the characteristic roots at $x=0$.
	Similarly, the roots of $L_{M}(y)=0$ are called the characteristic roots at $x=\infty$.
\end{defn}
\begin{rem}
	If $a$ is a characteristic root at $x=0$, then we obtain a solution of (\ref{geneeq}) in the form
	\begin{align}
		x^{\alpha}(1+O(x)),
	\end{align}
	where $a=q^{\alpha}$. 
	Also if $b$ is a characteristic root at $x=\infty$, then we obtain a solution of (\ref{geneeq}) in the form
	\begin{align}
		x^{\beta}(1+O(x^{-1})),
	\end{align}
	where $b=q^{\beta}$.
	In the general theory of linear $q$-difference equations, $\alpha$ and $-\beta$ are often called the characteristic exponents.
\end{rem}
For general theory of linear $q$-difference equations, if $q^{\alpha}$ and $q^{\alpha+n}$ are characteristic roots at $x=0$, where $n\in\mathbb{Z}_{>0}$, then we may need logarithmic terms for {some} solution.
For more details, see \cite{Ada}.
In this case, the equation (\ref{geneeq}) has the solution $x^{\alpha}\sum_{n=0}^{\infty}c_{n}x^{n}$ if this equation satisfies the condition
\begin{align}
	\sum_{m=0}^{n-1}c_{m}L_{n-i}(q^{\alpha+i})=0.
\end{align}
These can be considered as replacing $x=0$ and $x=\infty$.
In this paper, we will consider the equation that $a$ and $aq$ are the characteristic roots at $x=0$ or $x=\infty$ (i.e. $n=1$).
In this case, the non-logarithmic condition is simple:
\begin{lemm}\label{lemnonlog}
	For the equation (\ref{geneeq}), we have
	\begin{itemize}
		\item[(1)] $x=0$ is non-logarithmic and $a$, $aq$ are characteristic roots at $x=0$ if and only if $L_{1}(a)=0$.
		\item[(2)] $x=\infty$ is non-logarithmic and $a$, $aq^{-1}$ are characteristic roots at $x=\infty$ if and only if $L_{M-1}(a)=0$.
	\end{itemize}
\end{lemm}
\begin{rem}
	For more general case, see Proposition 3.1 in \cite{MY}.
\end{rem}
As an operator, the $q$-shift operator $T_{x}$ plays a similar role to the multiplication operator $x$, {i.e. $T_{x}x=qxT_{x}$}.
Thus  we can consider ``characteristic roots at $T_{x}=0$ and $T_{x}=\infty$''.
These roots are important for characterizing linear $q$-difference equations.
\begin{defn}
	The operator $\mathcal{L}$ can be rewritten as
	\begin{align}\label{geneeqP}
		\mathcal{L}=(P_{N}(x)T_{x}^{N}+\cdots+P_{0}(x)T_{x}^{0})T_{x}^{N'}.
	\end{align}
	The roots of the polynomial $P_{0}(x)$ (resp. $P_{N}(x)$) are called the characteristic roots at $T_{x}=0$ (resp. $T_{x}=\infty$).
\end{defn}
\begin{defn}\label{defconfi}
	Let $\{a_{j}\}_{1\leq j\leq N}$, $\{b_{j}\}_{1\leq j\leq N}$, $\{c_{i}\}_{1\leq i\leq M}$, $\{d_{i}\}_{1\leq i\leq M}$ be the characteristic roots of the equation (\ref{geneeq}) at $x=0$, $x=\infty$, $T_{x}=0$, $T_{x}=\infty$, respectively.
	Then Figure \ref{config} is called the configuration of (\ref{geneeq}).
	\begin{figure}[H]
		\centering
		\begin{tikzpicture}
			\draw (1,0)--(1,7);
			\draw [below] (1,0) node {$x=0$};
			\draw (6,0)--(6,7);
			\draw [below] (6,0) node {$x=\infty$};
			\draw (0,1)--(7,1);
			\draw [right] (7,1) node {$T_{x}=0$};
			\draw (0,6)--(7,6);
			\draw [right] (7,6) node {$T_{x}=\infty$};
			\filldraw (1,5) circle [radius=0.05];
			\draw [left] (1,5) node {$a_{1}$};
			\filldraw (1,4) circle [radius=0.05];
			\draw [left] (1,4) node {$a_{2}$};
			\filldraw (1,3) circle [radius=0.05];
			\draw [left] (1,3) node {$\vdots$};
			\filldraw (1,2) circle [radius=0.05];
			\draw [left] (1,2) node {$a_{N}$};
			\filldraw (6,5) circle [radius=0.05];
			\draw [right] (6,5) node {$b_{1}$};
			\filldraw (6,4) circle [radius=0.05];
			\draw [right] (6,4) node {$b_{2}$};
			\filldraw (6,3) circle [radius=0.05];
			\draw [right] (6,3) node {$\vdots$};
			\filldraw (6,2) circle [radius=0.05];
			\draw [right] (6,2) node {$b_{N}$};
			\filldraw (2,1) circle [radius=0.05];
			\draw [below] (2,1) node{$c_{1}$};
			\filldraw (3,1) circle [radius=0.05];
			\draw [below] (3,1) node{$c_{2}$};
			\filldraw (4,1) circle [radius=0.05];
			\draw [below] (4,1) node{$\cdots$};
			\filldraw (5,1) circle [radius=0.05];
			\draw [below] (5,1) node{$c_{M}$};
			\filldraw (2,6) circle [radius=0.05];
			\draw [above] (2,6) node{$d_{1}$};
			\filldraw (3,6) circle [radius=0.05];
			\draw [above] (3,6) node{$d_{2}$};
			\filldraw (4,6) circle [radius=0.05];
			\draw [above] (4,6) node{$\cdots$};
			\filldraw (5,6) circle [radius=0.05];
			\draw [above] (5,6) node{$d_{M}$};
		\end{tikzpicture}
		\caption{the configuration of the equation $\mathcal{L}f(x)=0$ (\ref{geneeq}).}
		\label{config}
	\end{figure}
\end{defn}
\begin{rem}
	We rewrite the operator $\mathcal{L}$ in two form (\ref{geneeqL}), (\ref{geneeqP}).
	By equating the coefficients, we find the relation $a_{1}\cdots a_{N} d_{1}\cdots d_{M}=b_{1}\cdots b_{N} c_{1}\cdots c_{M}$.
\end{rem}
\begin{rem}
	The name ``configuration'' comes from the point configuration of algebraic curves.
	The method to characterize $q$-difference equations by the configurations of algebraic curves was discussed in the context of $q$-Painlev\'{e} equation (cf. {\cite{KMNOY}, \cite{Ya}}).
	We use a quantization of this method by the quantum curve $\mathcal{L}=0$ to characterize the linear $q$-difference equation $\mathcal{L}f(x)=0$.
\end{rem}


\subsection{Variants of the $q$-hypergeometric equation}\label{subseccon}
In this subsection, we characterize variants of the $q$-hypergeometric equation introduced in \cite{HMST} by configurations.
As mentioned in section \ref{secintro}, the variants of the $q$-hypergeometric equation of degree two and degree three were defined by a special case of the $q$-Heun equation and the variant of the $q$-Heun equation of degree three, respectively.
We characterize not only the variants of the $q$-hypergeometric equation but also the $q$-Heun equation and the variant of the $q$-Heun equation of degree three.
The $q$-Heun equation was first introduced by \cite{Hahn1971}, and rediscovered by the eigenvalue problem for $A^{\langle 4\rangle}$ in {\cite{Takemura2017}}.
Here $A^{\langle i\rangle}$ is the $i$-th degenerated Ruijsenaars-van Diejen operator of one variable.
Also the variant of the $q$-Heun equation of degree three was defined as the eigenvalue problem for $A^{\langle3\rangle}$ in \cite{Takemura2018}.
In this paper we do not discuss {about} the Ruijsenaars-van Diejen operator.
For more details, see {\cite{Ru,vD}}.
We introduce the $q$-Heun equation and the variant of the $q$-Heun equation of degree three by the forms of {\cite{Takemura2018}}.
\begin{defn}[{\cite{Takemura2018}}]
	The $q$-Heun equation is defined by
	\begin{align}\label{defqheun}
		(A^{\langle 4\rangle}-E)f(x)=0,
	\end{align}
	where
	\begin{align}
		\notag A^{\langle4\rangle}=&x^{-1}\prod_{i=1}^{2}(x-q^{h_{i}+1/2}t_{i})  \cdot T_{x}^{-1}+q^{\alpha_{1}+\alpha_{2}}x^{-1}\prod_{i=1}^{2}(x-q^{l_{i}-1/2}t_{i}) \cdot T_{x}\\
		&-(q^{\alpha_{1}}+q^{\alpha_{2}})x+q^{(h_{1}+h_{2}+l_{1}+l_{2}+\alpha_{1}+\alpha_{2})/2}(q^{\beta/2}+q^{-\beta/2})t_{1}t_{2}x^{-1},
	\end{align}
	and the variant of $q$-Heun equation of degree three is defined by
	\begin{align}\label{defqheun3}
		(A^{\langle3\rangle}-E)f(x)=0,
	\end{align}
	where
	\begin{align}
		\notag A^{\langle3\rangle}=&x^{-1}\prod_{i=1}^{3}(x-q^{h_{i}+1/2}t_{i}) \cdot T_{x}^{-1}+x^{-1}\prod_{i=1}^{3}(x-q^{l_{i}-1/2}t_{i}) \cdot T_{x}\\
		&-(q^{1/2}+q^{-1/2})x^{2}+\sum_{i=1}^{3}(q^{h_{i}}+q^{l_{i}})t_{i}x+q^{(l_{1}+l_{2}+l_{3}+h_{1}+h_{2}+h_{3})/2}(q^{\beta/2}+q^{-\beta/2})t_{1}t_{2}t_{3}x^{-1}.
	\end{align}
\end{defn}
We can easily find the configurations of the equations (\ref{defqheun}) and (\ref{defqheun3}).
The configuration of (\ref{defqheun}) is given as follows.
\begin{figure}[H]
	\centering
	\begin{tikzpicture}
		\draw (1,0)--(1,6);
		\draw [below] (1,0) node {$x=0$};
		\draw (6,0)--(6,6);
		\draw [below] (6,0) node {$x=\infty$};
		\draw (0,1)--(7,1);
		\draw [right] (7,1) node {$T_{x}=0$};
		\draw (0,5)--(7,5);
		\draw [right] (7,5) node {$T_{x}=\infty$};
		\filldraw (1,2) circle [radius=0.05];
		\draw [left] (1,2) node {$q^{\lambda_{-}}$};
		\filldraw (1,4) circle [radius=0.05];
		\draw [left] (1,4) node {$q^{\lambda_{+}}$};
		\filldraw (6,2) circle [radius=0.05];
		\draw [right] (6,2) node {$q^{-\alpha_{2}}$};
		\filldraw (6,4) circle [radius=0.05];
		\draw [right] (6,4) node {$q^{-\alpha_{1}}$};
		\filldraw (2,1) circle [radius=0.05];
		\draw [below] (2,1) node{$q^{h_{1}+1/2}t_{1}$};
		\filldraw (5,1) circle [radius=0.05];
		\draw [below] (5,1) node{$q^{h_{2}+1/2}t_{2}$};
		\filldraw (2,5) circle [radius=0.05];
		\draw [above] (2,5) node{$q^{l_{1}-1/2}t_{1}$};
		\filldraw (5,5) circle [radius=0.05];
		\draw [above] (5,5) node{$q^{l_{2}-1/2}t_{2}$};
	\end{tikzpicture}
	\caption{the configuration of the equation {$(A^{\langle4\rangle}-E)f(x)=0$} (\ref{defqheun}).}
	\label{conqheun}
\end{figure}
Here, $\lambda_{\pm}=\frac{1}{2}(h_{1}+h_{2}-l_{1}-l_{2}-\alpha_{1}-\alpha_{2}\pm\beta)$. 
Also, the configuration of (\ref{defqheun3}) is given as follows.
\begin{figure}[H]
	\centering
	\begin{tikzpicture}
		\draw (1,0)--(1,6);
		\draw [below] (1,0) node {$x=0$};
		\draw (9,0)--(9,6);
		\draw [below] (9,0) node {$x=\infty$};
		\draw (0,1)--(10,1);
		\draw [right] (10,1) node {$T_{x}=0$};
		\draw (0,5)--(10,5);
		\draw [right] (10,5) node {$T_{x}=\infty$};
		\filldraw (1,2) circle [radius=0.05];
		\draw [left] (1,2) node {$q^{\mu_{-}}$};
		\filldraw (1,4) circle [radius=0.05];
		\draw [left] (1,4) node {$q^{\mu_{+}}$};
		\filldraw (9,3) circle [radius=0.05];
		\draw (9,3) circle [radius=0.1];
		\draw [right] (9,3) node {$q^{1/2}, q^{-1/2}$};
		\filldraw (2,1) circle [radius=0.05];
		\draw [below] (2,1) node{$q^{h_{1}+1/2}t_{1}$};
		\filldraw (5,1) circle [radius=0.05];
		\draw [below] (5,1) node{$q^{h_{2}+1/2}t_{2}$};
		\filldraw (8,1) circle [radius=0.05];
		\draw [below] (8,1) node{$q^{h_{3}+1/2}t_{3}$};
		\filldraw (2,5) circle [radius=0.05];
		\draw [above] (2,5) node{$q^{l_{1}-1/2}t_{1}$};
		\filldraw (5,5) circle [radius=0.05];
		\draw [above] (5,5) node{$q^{l_{2}-1/2}t_{2}$};
		\filldraw (8,5) circle [radius=0.05];
		\draw [above] (8,5) node{$q^{l_{3}-1/2}t_{3}$};
	\end{tikzpicture}
	\caption{the configuration of the equation {$(A^{\langle3\rangle}-E)f(x)=0$} (\ref{defqheun3}).}
	\label{conqheun3}
\end{figure}
Here, $\mu_{\pm}=\frac{1}{2}(h_{1}+h_{2}+h_{3}-l_{1}-l_{2}-l_{3}+3\pm\beta)$.
The double point \tikz{\filldraw (0,0) circle [radius=0.05]; \draw (0,0) circle [radius=0.1];} in Figure \ref{conqheun3} means a non-logarithmic singularity which has characteristic roots $a$, $aq$.
In this paper, we use the double point in the sense as above.
In fact, figure \ref{conqheun} and \ref{conqheun3} characterize the equations (\ref{defqheun}) and (\ref{defqheun3}), respectively.
\begin{prop}\label{propconqheun}
	We consider the $q$-difference equations
	\begin{align}
		&\mathcal{L}_{2}f(x)=\sum_{i=-1}^{1}\sum_{j=-1}^{1}a_{i,j}x^{i}T_{x}^{j}f(x)=0,\\
		&\mathcal{L}_{3}f(x)=\sum_{i=-1}^{2}\sum_{j=-1}^{1}a_{i,j}x^{i}T_{x}^{j}f(x)=0.
	\end{align}
	\begin{itemize}
		\item[(1)] If the configuration of $\mathcal{L}_{2}f(x)=0$ is given by Figure \ref{conqheun}, then the equation $\mathcal{L}_{2}f(x)=0$ coincides with the $q$-Heun equation (\ref{defqheun}) up to the eigenvalue $E$.
		\item[(2)] If the configuration of $\mathcal{L}_{3}f(x)=0$ is given by Figure \ref{conqheun3}, then the equation $\mathcal{L}_{3}f(x)=0$ coincides with the variant of the $q$-Heun equation of degree three (\ref{defqheun3}) up to the eigenvalue $E$.
	\end{itemize}
\end{prop}
\begin{proof}
	We prove only (2) because we can show (1) in the same way as the proof of (2).
	We assume that the configuration of $\mathcal{L}_{3}f(x)=0$ is given by Figure \ref{conqheun3}.
	Since the characteristic roots at $x=0$ are $q^{\mu_{+}}$, $q^{\mu_{-}}$, the ratio of $a_{-1,-1}$, $a_{-1,0}$, $a_{-1,1}$ is determined.
	More precisely, we have $a_{-1,-1}T_{x}^{-1}+a_{-1,0}+a_{-1,1}T_{x}=a_{-1,1}T_{x}^{-1}(T_{x}-q^{\mu_{+}})(T_{x}-q^{\mu_{-}})$.
	In the same way the ratios of $a_{1,-1}$, $a_{1,0}$, $a_{1,1}$, $a_{-1,-1}$, $a_{0,-1}$, $a_{1,-1}$, $a_{2,-1}$, $a_{-1,1}$, $a_{0,1}$, $a_{1,1}$, $a_{2,1}$ are determined by the characteristic roots.
	Using Lemma \ref{lemnonlog} (2), we have $a_{1,-1}q^{-1/2}+a_{1,0}+a_{1,1}q^{1/2}=0$.
	Since the ratio of $a_{1,-1}$, $a_{1,1}$ is already determined, we obtain the ratio of $a_{1,-1}$, $a_{1,0}$.
	Thus we find the ratio of $\{a_{i,j}\mid (i,j)\neq(0,0)\}$.
	Therefore $\mathcal{L}_{3}$ coincides with the equation (\ref{defqheun3}) up to the eigenvalue $E$.
\end{proof}
\begin{rem}
	For equations (\ref{defqheun}) and (\ref{defqheun3}), the eigenvalue $E$ is independent of their configurations.
	It is reasonable to regard $E$ as an accessory parameter.
\end{rem}
Next, we introduce the variants of the $q$-difference equation defined by \cite{HMST}.
The variant of the $q$-hypergeometric equation of degree two (resp. degree three) was defined by specializing the $q$-Heun equation (resp. the variant of the $q$-Heun equation of degree three) so that the point $x=0$ is essentially non-singular.
\begin{defn}[{\cite{HMST}}]
	The variant of the $q$-hypergeometric equation of degree two is defined by
	\begin{align}\label{defqhyp2}
		\mathcal{H}_{2}f(x)=0,
	\end{align}
	where
	\begin{align}
		\mathcal{H}_{2}=&\prod_{i=1}^{2}(x-q^{h_{i}+1/2}t_{i})\cdot T_{x}^{-1}+q^{\alpha_{1}+\alpha_{2}}\prod_{i=1}^{2}(x-q^{l_{i}-1/2}t_{i})\cdot T_{x}-(q^{\alpha_{1}}+q^{\alpha_{2}})x^{2}\notag\\
		&+Ex+p(q^{1/2}+q^{-1/2})t_{1}t_{2},\\
		p=&q^{(h_{1}+h_{2}+l_{1}+l_{2}+\alpha_{1}+\alpha_{2})/2},\ \ E=-p\{(q^{-h_{2}}+q^{-l_{2}})t_{1}+(q^{-h_{1}}+q^{-l_{1}})t_{2}\}.
	\end{align}
	The variant of the $q$-hypergeometric equation of degree three is defined by
	\begin{align}\label{defqhyp3}
		\mathcal{H}_{3}f(x)=0,
	\end{align}
	where
	\begin{align}
		\notag \mathcal{H}_{3}=&\prod_{i=1}^{3}(x-q^{h_{i}+1/2}t_{i})\cdot T_{x}^{-1}+q^{2\alpha+1}\prod_{i=1}^{3}(x-q^{l_{i}-1/2}t_{i})\cdot T_{x}\\
		\notag &+q^{\alpha}\biggl[-(q+1)x^{3}+q^{1/2}\sum_{i=1}^{3}(q^{h_{i}}+q^{l_{i}})t_{i}x^{2}\\
		\notag &-q^{(h_{1}+h_{2}+h_{3}+l_{1}+l_{2}+l_{3}+1)/2}t_{1}t_{2}t_{3}\sum_{i=1}^{3}((q^{-h_{i}}+q^{-l_{i}})/t_{i})x\\
		&+q^{(h_{1}+h_{2}+h_{3}+l_{1}+l_{2}+l_{3})/2}(q+1)t_{1}t_{2}t_{3}\biggr].
	\end{align}
\end{defn}
We can easily check that Figures \ref{conqhyp2} and \ref{conqhyp3} are the configurations of the equations (\ref{defqhyp2}) and (\ref{defqhyp3}), respectively.
Here, $\lambda_{0}=\frac{1}{2}(h_{1}+h_{2}-l_{1}-l_{2}-\alpha_{1}-\alpha_{2}+1)$ and $\nu=\frac{1}{2}(h_{1}+h_{2}+h_{3}-l_{1}-l_{2}-l_{3}+1)$.
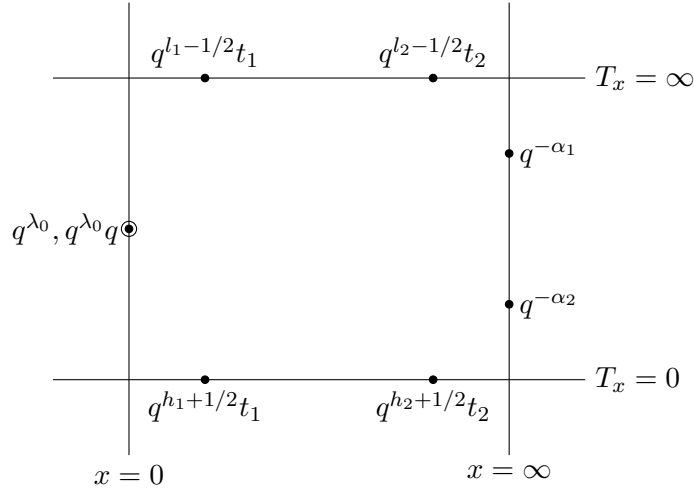
\begin{figure}[H]
	\centering
	\begin{tikzpicture}
		\draw (1,0)--(1,6);
		\draw [below] (1,0) node {$x=0$};
		\draw (6,0)--(6,6);
		\draw [below] (6,0) node {$x=\infty$};
		\draw (0,1)--(7,1);
		\draw [right] (7,1) node {$T_{x}=0$};
		\draw (0,5)--(7,5);
		\draw [right] (7,5) node {$T_{x}=\infty$};
		\filldraw (1,3) circle [radius=0.05];
		\draw [left] (1,3) node {$q^{\lambda_{0}},q^{\lambda_{0}}q$};
		\draw (1,3) circle [radius=0.1];
		\filldraw (6,2) circle [radius=0.05];
		\draw [right] (6,2) node {$q^{-\alpha_{2}}$};
		\filldraw (6,4) circle [radius=0.05];
		\draw [right] (6,4) node {$q^{-\alpha_{1}}$};
		\filldraw (2,1) circle [radius=0.05];
		\draw [below] (2,1) node{$q^{h_{1}+1/2}t_{1}$};
		\filldraw (5,1) circle [radius=0.05];
		\draw [below] (5,1) node{$q^{h_{2}+1/2}t_{2}$};
		\filldraw (2,5) circle [radius=0.05];
		\draw [above] (2,5) node{$q^{l_{1}-1/2}t_{1}$};
		\filldraw (5,5) circle [radius=0.05];
		\draw [above] (5,5) node{$q^{l_{2}-1/2}t_{2}$};
	\end{tikzpicture}
	\caption{the configuration of the equation $\mathcal{H}_{2}f(x)=0$ (\ref{defqhyp2}).}
	\label{conqhyp2}
\end{figure}
\begin{figure}[H]
	\centering
	\begin{tikzpicture}
		\draw (1,0)--(1,6);
		\draw [below] (1,0) node {$x=0$};
		\draw (9,0)--(9,6);
		\draw [below] (9,0) node {$x=\infty$};
		\draw (0,1)--(10,1);
		\draw [right] (10,1) node {$T_{x}=0$};
		\draw (0,5)--(10,5);
		\draw [right] (10,5) node {$T_{x}=\infty$};
		\filldraw (1,3) circle [radius=0.05];
		\draw [left] (1,3) node {$q^{\nu-\alpha},q^{\nu-\alpha+1}$};
		\draw (1,3) circle [radius=0.1];
		\filldraw (9,3) circle [radius=0.05];
		\draw (9,3) circle [radius=0.1];
		\draw [right] (9,3) node {$q^{-\alpha}, q^{-\alpha-1}$};
		\filldraw (2,1) circle [radius=0.05];
		\draw [below] (2,1) node{$q^{h_{1}+1/2}t_{1}$};
		\filldraw (5,1) circle [radius=0.05];
		\draw [below] (5,1) node{$q^{h_{2}+1/2}t_{2}$};
		\filldraw (8,1) circle [radius=0.05];
		\draw [below] (8,1) node{$q^{h_{3}+1/2}t_{3}$};
		\filldraw (2,5) circle [radius=0.05];
		\draw [above] (2,5) node{$q^{l_{1}-1/2}t_{1}$};
		\filldraw (5,5) circle [radius=0.05];
		\draw [above] (5,5) node{$q^{l_{2}-1/2}t_{2}$};
		\filldraw (8,5) circle [radius=0.05];
		\draw [above] (8,5) node{$q^{l_{3}-1/2}t_{3}$};
	\end{tikzpicture}
	\caption{the configuration of the equation $\mathcal{H}_{3}f(x)=0$ (\ref{defqhyp3}).}
	\label{conqhyp3}
\end{figure}
These configurations characterize the equations (\ref{defqhyp2}) and (\ref{defqhyp3}).
\begin{prop}\label{propconqhyp}
	We consider the $q$-difference equations
	\begin{align}
		&\mathcal{L}_{2}f(x)=\sum_{i=0}^{2}\sum_{j=-1}^{1}a_{i,j}x^{i}T_{x}^{j}f(x)=0,\\
		&\mathcal{L}_{3}f(x)=\sum_{i=0}^{3}\sum_{j=-1}^{1}a_{i,j}x^{i}T_{x}^{j}f(x)=0.
	\end{align}
	\begin{itemize}
		\item[(1)] If the configuration of $\mathcal{L}_{2}f(x)=0$ is given by Figure \ref{conqhyp2}, then the equation $\mathcal{L}_{2}f(x)=0$ coincides with the variant of the $q$-hypergeometric equation of degree two (\ref{defqhyp2}), that is, $\mathcal{L}_{2}=\alpha\mathcal{H}_{2}$ ($\alpha\in\mathbb{C}$). 
		\item[(2)] If the configuration of $\mathcal{L}_{3}f(x)=0$ is given by Figure \ref{conqhyp3}, then the equation $\mathcal{L}_{3}f(x)=0$ coincides with the variant of the $q$-hypergeometric equation of degree three (\ref{defqhyp3}), that is, $\mathcal{L}_{3}=\alpha\mathcal{H}_{3}$ ($\alpha\in\mathbb{C}$).
	\end{itemize}
\end{prop}
\begin{proof}
	We prove (2) only because (1) is proved in the same way.
	We assume that the configuration of $\mathcal{L}_{3}f(x)=0$ is given by Figure \ref{conqhyp3}.
	Similar to the proof of Proposition \ref{propconqheun}, the ratio of $\{a_{i,j}\mid (i,j)\neq (0,0),(1,0)\}$ is determine by the characteristic roots.
	Also, by the non-logarithmic condition at $x=0$ and $x=\infty$, we find the ratios of $a_{0,-1}$ and $a_{0,0}$, $a_{1,-1}$ and $a_{1,0}$, respectively.
	Thus the ratios of $\{a_{i,j}\}$ are determined.
	This completes the proof of (2).
\end{proof}
\begin{rem}
	In \cite{HMST}, the equation $\mathcal{H}_{3}f(x)=0$ is a $q$-analog of the Fuchsian differential equation, which has the following Riemann scheme:
	\begin{align}\label{Riemannscheme3}
		\left\{\begin{array}{ccccc}
			x=0& x=t_{1} & x=t_{2} & x=t_{3} & x=\infty\\
			\nu-\alpha& 0 & 0 & 0 & \alpha\\
			\nu-\alpha+1& l_{1}-h_{1}& l_{2}-h_{2}& l_{3}-h_{3}&\alpha+1
		\end{array}\right\}.
	\end{align}
	Here, $x=0$ and $x=\infty$ are essentially non-singular.
	By some gauge transformation, this equation becomes a Fuchsian differential equation with three singularities $\{t_{1},t_{2},t_{3}\}$.
	In this sense, the variant of the $q$-hypergeometric equation of degree three is a $q$-analog of the Riemann-Papperitz differential equation {\ref{RiePapeq}}.
	Also, the equation $\mathcal{H}_{2}f(x)=0$ is a $q$-analog of the Fuchsian differential equation, which has the following Riemann scheme:
	\begin{align}\label{Riemannscheme2}
		\left\{\begin{array}{cccc}
			x=0& x=t_{1} & x=t_{2} & x=\infty\\
			\lambda_{0}& 0 & 0  & \alpha_{1}\\
			\lambda_{0}+1& l_{1}-h_{1}& l_{2}-h_{2}& \alpha_{2}
		\end{array}\right\},
	\end{align}
	where $x=0$ is essentially non-singular.
\end{rem}
\begin{rem}\label{Remdegeeq}
	Taking the limit $t_{3}\to\infty$, the equation $\mathcal{H}_{3}f(x)=0$ (\ref{defqhyp3}) becomes the equation $\mathcal{H}_{2}f(x)=0$ (\ref{defqhyp2}) with the parameters $(\alpha_{1},\alpha_{2})=(\alpha,\alpha-h_{3}+l_{3})$.
	Also, taking the limit $t_{2}\to0$, the equation (\ref{defqhyp2}) becomes the equation
	\begin{align}
		&\mathcal{H}_{1}f(x)=0,\\
		&\mathcal{H}_{1}=(x-q^{h_{1}+1/2}t_{1})T_{x}^{-1}+q^{\alpha_{1}+\alpha_{2}}(x-q^{l_{1}-1/2}t_{1})T_{x}\notag\\
		&-[(q^{\alpha_{1}+\alpha_{2}})x-q^{(h_{1}+h_{2}+l_{1}+l_{2}+\alpha_{1}+\alpha_{2})/2}(q^{-h_{2}}+q^{-l_{2}})t_{1}].
	\end{align} 
	By setting 
	\begin{align}
		t_{1}=1,\ \ h_{1}=1/2,\ \ h_{2}-l_{2}=\alpha_{1}+\alpha_{2}+l_{1}-3/2,\ \ a=q^{\alpha_{1}},\ \ b=q^{\alpha_{2}},\ \ c=q^{\alpha_{1}+\alpha_{2}+l_{1}-1/2},
	\end{align}
	Heine's $q$-hypergeometric equation (\ref{Heineeq}) is realized.
	For more details, see \cite{HMST}.
	These degenerations can be considered by the configuration as follows:
	
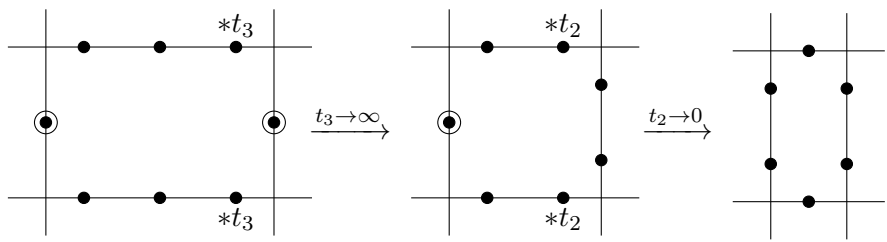
\begin{center}
		\begin{tikzpicture}[scale=0.5]
			\draw (1,0)--(1,6);
			\draw (7,0)--(7,6);
			\draw (0,1)--(8,1);
			\draw (0,5)--(8,5);
			\filldraw (1,3) circle [radius=0.15];
			\draw (1,3) circle [radius=0.3];
			\filldraw (7,3) circle [radius=0.15];
			\draw (7,3) circle [radius=0.3];
			\filldraw (2,1) circle [radius=0.15];
			\filldraw (4,1) circle [radius=0.15];
			\filldraw (6,1) circle [radius=0.15];
			\draw [below] (6,1) node {$\ast t_{3}$};
			\filldraw (2,5) circle [radius=0.15];
			\filldraw (4,5) circle [radius=0.15];
			\filldraw (6,5) circle [radius=0.15];
			\draw [above] (6,5) node {$\ast t_{3}$};
			\draw (9,3) node{$\xrightarrow{t_{3}\to\infty}$};
		\end{tikzpicture}
		\begin{tikzpicture}[scale=0.5]
			\draw (1,0)--(1,6);
			\draw (5,0)--(5,6);
			\draw (0,1)--(6,1);
			\draw (0,5)--(6,5);
			\filldraw (1,3) circle [radius=0.15];
			\draw (1,3) circle [radius=0.3];
			\filldraw (5,2) circle [radius=0.15];
			\filldraw (5,4) circle [radius=0.15];
			\filldraw (2,1) circle [radius=0.15];
			\filldraw (4,1) circle [radius=0.15];
			\draw [below] (4,1) node {$\ast t_{2}$};
			\filldraw (2,5) circle [radius=0.15];
			\filldraw (4,5) circle [radius=0.15];
			\draw [above] (4,5) node {$\ast t_{2}$};
			\draw (7,3) node {$\xrightarrow{t_{2}\to0}$};
		\end{tikzpicture}
		\begin{tikzpicture}[scale=0.5]
			\draw (1,0)--(1,6);
			\draw (3,0)--(3,6);
			\draw (0,1)--(4,1);
			\draw (0,5)--(4,5);
			\filldraw (1,2) circle [radius=0.15];
			\filldraw (1,4) circle [radius=0.15];
			\filldraw (3,2) circle [radius=0.15];
			\filldraw (3,4) circle [radius=0.15];
			\filldraw (2,1) circle [radius=0.15];
			\filldraw (2,5) circle [radius=0.15];
		\end{tikzpicture}
	\end{center}
	Here, the right configuration characterize Heine's equation.
	More precisely, the configuration of Heine's equation $[x(1-aT_{x})(1-bT_{x})-(1-T_{x})(1-cq^{-1}T_{x})]f(x)=0$ is as follows:
	
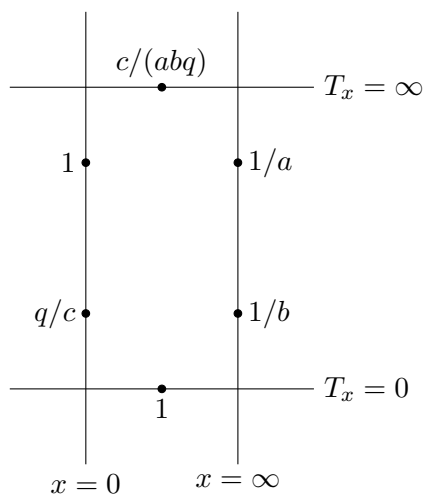
\begin{center}
		\begin{tikzpicture}
			\draw (1,0)--(1,6);
			\draw [below] (1,0) node {$x=0$};
			\draw (3,0)--(3,6);
			\draw [below] (3,0) node {$x=\infty$};
			\draw (0,1)--(4,1);
			\draw [right] (4,1) node {$T_{x}=0$};
			\draw (0,5)--(4,5);
			\draw [right] (4,5) node {$T_{x}=\infty$};
			\filldraw (1,2) circle [radius=0.05];
			\draw [left] (1,2) node {$q/c$};
			\filldraw (1,4) circle [radius=0.05];
			\draw [left] (1,4) node {1};
			\filldraw (3,2) circle [radius=0.05];
			\draw [right] (3,2) node {$1/b$};
			\filldraw (3,4) circle [radius=0.05];
			\draw [right] (3,4) node {$1/a$};
			\filldraw (2,1) circle [radius=0.05];
			\draw [below] (2,1) node {1};
			\filldraw (2,5) circle [radius=0.05];
			\draw [above] (2,5) node {$c/(abq)$};
		\end{tikzpicture}
	\end{center}
	In differential case, these degenerations mean the limit of Fuchsian differential equation with three singularities $\{t_{1},t_{2},t_{3}\}$, that is, $\{t_{1},t_{2},t_{3}\}\to\{t_{1},t_{2},\infty\}\to\{0,t_{1},\infty\}$.
\end{rem}
\begin{rem}
	In Proposition \ref{propconqhyp} we find that the variants of the $q$-hypergeometric equation (\ref{defqhyp2}), (\ref{defqhyp3}) are rigid by the configuration.
	This is corresponding to the Riemann-Papperitz equation {(\ref{RiePapeq})} being rigid by the Riemann scheme (\ref{RiePapsch}).
\end{rem}

\section{Integral solutions}\label{secint}
In this section, we show integral solutions for the variants of the $q$-hypergeometric equation.
First, we derive a $q$-difference equation for the Jackson integral of Jordan-Pochhammer type.
Next, we show integral solutions for the variant of the $q$-hypergeometric equation of degree three $\mathcal{H}_{3}f(x)=0$ by considering a special case of the Jordan-Pochhammer integral. 
Integral solutions for the variant of the $q$-hypergeometric equation of degree two are obtained in a similar way for some limit of the integral solutions for $\mathcal{H}_{3}f(x)=0$.
A linear $q$-difference system associated with the Jackson integral of Jordan-Pochhammer type was obtained in \cite{Ma}, \cite{Mi1989}.
A $q$-difference system for the Jackson integral of Selberg type, which contains the system for the Jordan-Pochhammer type, was obtained in \cite{Mi1994}.
In \cite{It}, a $q$-difference system for the Jackson integral of Selberg type was also discussed, and above results were summarized.
Hence, for more details of the Jackson integral of Jordan-Pochhammer type, see \cite{It,Ma,Mi1989,Mi1994}.
In this paper, we derive a $q$-difference equation for the Jackson integral of Jordan-Pochhammer type by integrating the equation that the integrand satisfies.
The configuration of our equation can be calculated more easily than the configuration of an equation which is derived {from} the above system.

\begin{defn}[{\cite{Ja}}]\label{defJac}
	The Jackson integrals of the function $f(t)$ are defined as follows:
	\begin{align}
		&\int_{0}^{\tau} f(t)\frac{d_{q}t}{t}=(1-q)\sum_{n=0}^{\infty}f(\tau q^{n}),\\
		&\int_{0}^{\tau\infty}f(t)\frac{d_{q}t}{t}=(1-q)\sum_{n=-\infty}^{\infty}f(\tau q^{n}),\\
		&\int_{\tau_{1}}^{\tau_{2}}f(t)\frac{d_{q}t}{t}=\int_{0}^{\tau_{2}}f(t)\frac{d_{q}t}{t}-\int_{0}^{\tau_{1}}f(t)\frac{d_{q}t}{t}.
	\end{align}
\end{defn}
\begin{defn}[\cite{Mi1989}]\label{defJP}
	Let $\psi$ be the function
	\begin{align}\label{integrandJP}
		\psi(x,t)=t^{\alpha}\frac{(Axt,a_{1}t,a_{2}t,\ldots,a_{M}t)_{\infty}}{(Bxt,b_{1}t,b_{2}t,\ldots,b_{M}t)_{\infty}}.
	\end{align}
	The Jackson integral of Jordan-Pochhammer type is defined by
	\begin{align}\label{JP}
		\varphi(x,\tau)=\int_{0}^{\tau\infty}\psi(x,t)\frac{d_{q}t}{t}.
	\end{align}
\end{defn}
Lemma \ref{lemeqpsi} plays an essential role in deriving the $q$-difference equation which the Jackson integral of Jordan-Pochhammer type satisfies.
\begin{lemm}\label{lemeqpsi}
	The integrand (\ref{integrandJP}) satisfies the following equation
	\begin{align}\label{eqpsi}
		\sum_{k=0}^{M}&(-1)^{k}x^{M-k}(e_{k}(a)T_{t}T_{x}^{-1}-q^{\alpha}e_{k}(b))\notag\\
		&\times(B-AT_{x})\cdots(B-Aq^{M-k-1}T_{x})(1-q^{-(k-1)}T_{x})\cdots(1-T_{x})\psi=0.
	\end{align}
	{Here, $e_{i}$ is the elementary symmetric polynomial of degree $i$.}
\end{lemm}
\begin{proof}
	The equations
	\begin{align}
		\label{preqpsi1}&xt(B-AT_{x})\psi=(1-T_{x})\psi,\\
		\label{preqpsi2}&\prod_{i=1}^{M}(1-a_{i}t)\cdot T_{t}T_{x}^{-1}\psi=q^{\alpha}\prod_{i=1}^{M}(1-b_{i}t)\cdot \psi,
	\end{align}
	can be verified by direct calculations.
	Using the elementary symmetric functions, the equation (\ref{preqpsi2}) can be rewitten as
	\begin{align}\label{preqpsi3}
		\sum_{k=0}^{M}(-t)^{k}[e_{k}(a)T_{t}T_{x}^{-1}-q^{\alpha}e_{k}(b)]\psi=0.
	\end{align}
	Since $[(xt)(B-AT_{x})]^{k}=(xt)^{k}(B-AT_{x})(B-AqT_{x})\cdots(B-Aq^{k-1}T_{x})$ holds, we have
	\begin{align}
		\notag&x^{M}(B-AT_{x})(B-AqT_{x})\cdots(B-Aq^{M-1}T_{x})t^{k}(T_{t}T_{x}^{-1})^{\varepsilon}\psi\\
		\notag=&x^{M-k}(xt)^{k}(T_{t}T_{x}^{-1})^{\varepsilon}(B-Aq^{k}T_{x})\cdots(B-Aq^{M-1}T_{x})(B-AT_{x})\cdots(B-Aq^{k-1}T_{x})\psi\\
		\notag=&x^{M-k}(T_{t}T_{x}^{-1})^{\varepsilon}(B-AT_{x})\cdots(B-Aq^{M-1-k}T_{x})(xt)^{k}(B-AT_{x})\cdots(B-Aq^{k-1}T_{x})\psi\\
		=&x^{M-k}(T_{t}T_{x}^{-1})^{\varepsilon}(B-AT_{x})\cdots(B-Aq^{M-1-k}T_{x})[xt(B-AT_{x})]^{k}\psi,
	\end{align}
	for $\varepsilon=0,1$. By using the equation (\ref{preqpsi1}), we get
	\begin{align}
		[(xt)(B-AT_{x})]^{k}\psi=(1-q^{-(k-1)}T_{x})\cdots(1-T_{x})\psi.
	\end{align}
	Therefore, multiplying the equation (\ref{preqpsi3}) by $x^{M}(B-AT_{x})(B-AqT_{x})\cdots(B-Aq^{M-1}T_{x})$ yields the desired equation (\ref{eqpsi}).
\end{proof}
If the integral path (the period of integration) is chosen appropriately, then we can derive the $q$-difference equation that the Jackson integral of Jordan-Pochhammer type satisfies.
\begin{prop}\label{propeqJP}
	Assuming $T_{x}\tau=q^{l}\tau$ for $l\in\mathbb{Z}$, the integral (\ref{JP}) satisfies the $q$-difference equation
	\begin{align}\label{eqJP}
		\sum_{k=0}^{M}\bigl[&(-1)^{k}x^{M-k}(e_{k}(a)T_{x}^{-1}-q^{\alpha}e_{k}(b))(B-AT_{x})\cdots(B-Aq^{M-k-1}T_{x})\notag\\
		&\times(1-q^{-(k-1)}T_{x})\cdots(1-T_{x})\bigr]\varphi=0.
	\end{align}
\end{prop}
\begin{proof}
	Since $T_{x}\tau=q^{l}\tau$, we get
	\begin{align}
		\int_{0}^{\tau\infty}(T_{t}^{i}T_{x}^{j}\psi(x,t))\frac{d_{q}t}{t}&=(1-q)\sum_{n\in\mathbb{Z}}\psi(T_{x}^{i}x,q^{n+j}\tau)\notag\\
		&=(1-q)\sum_{n\in\mathbb{Z}}\psi(T_{x}^{i}x,q^{n+j-il}T_{x}^{i}\tau)\notag\\
		&=(1-q)\sum_{n\in\mathbb{Z}}\psi(T_{x}^{i}x,q^{n}T_{x}^{i}\tau)\notag\\
		&=(1-q)\sum_{n\in\mathbb{Z}}T_{x}^{i}\psi(x,q^{n}\tau)\notag\\
		&=T_{x}^{i}\int_{0}^{\tau\infty}\psi(x,t)\frac{d_{q}t}{t}.
	\end{align}
	Therefore we have the desired equation (\ref{eqJP}) by integrating the equation (\ref{eqpsi}).
\end{proof}
\begin{rem}\label{remM1case}
	We put $M=1$, then the equation (\ref{eqJP}) is Heine's $q$-hypergeometric equation.
	With the changes of the variable and parameters
	\begin{align}\label{changeM1}
		z=\frac{qBx}{a_{1}},\ \ a=q^{\alpha},\ \ b=\frac{A}{B},\ \ c=q^{\alpha+1}\frac{b_{1}}{a_{1}},
	\end{align}
	we find
	\begin{align}
		&[x(T_{x}^{-1}-q^{\alpha})(B-AT_{x})-(a_{1}T_{x}^{-1}-q^{\alpha}b_{1})(1-T_{x})]\notag\\
		&=T_{z}^{-1}[z(1-aT_{z})(1-bT_{z})-(1-cq^{-1}T_{z})(1-T_{z})].
	\end{align}
	Therefore we have integral solutions $\displaystyle\int_{0}^{\tau\infty}t^{\alpha}\frac{(Axt,a_{1}t)_{\infty}}{(Bxt,b_{1}t)_{\infty}}\frac{d_{q}t}{t}$ for Heine's equation.
\end{rem}
We will reduce the special case of (\ref{eqJP}) to the variant of the $q$-hypergeometric equation of degree three, which yields the integral solutions.
We put $M=3$, $\alpha=1$ and $Aa_{1}a_{2}a_{3}=q^{2}Bb_{1}b_{2}b_{3}$.
Then the equation (\ref{eqJP}) can be written as
\begin{align}\label{eqJP3}
	[&x^{3}(1-qT_{x})(B-AT_{x})(B-AqT_{x})(B-Aq^{2}T_{x})\notag\\
	-&x^{2}(e_{1}(a)-qe_{1}(b)T_{x})(B-AT_{x})(B-AqT_{x})(1-T_{x})\notag\\
	+&x(e_{2}(a)-qe_{2}(b)T_{x})(B-AT_{x})(1-q^{-1}T_{x})(1-T_{x})\notag\\
	-&e_{3}(a)B^{-1}(B-q^{-1}AT_{x})(1-q^{-2}T_{x})(1-q^{-1}T_{x})(1-T_{x})]T_{x}^{-1}\varphi=0,
\end{align}
so that the equation becomes
\begin{align}
	(1-q^{-2}T_{x})(B-Aq^{-1}T_{x})\mathcal{E}_{3}\varphi=0,
\end{align}
where $\mathcal{E}_{3}$ is the $q$-difference operator
\begin{align}\label{defE3}
	\mathcal{E}_{3}=&[x^{3}(B-AT_{x})(B-AqT_{x})-x^{2}(e_{1}(a)-qe_{1}(b)T_{x})(B-AT_{x})\notag\\
	&+x(e_{2}(a)-qe_{2}(b)T_{x})(1-T_{x})-e_{3}(a)B^{-1}(1-q^{-1}T_{x})(1-T_{x})]T_{x}^{-1}.
\end{align}
The configuration of (\ref{defE3}) is given as follows.
\begin{figure}[H]
	\centering
	\begin{tikzpicture}
		\draw (1,0)--(1,6);
		\draw [below] (1,0) node {$x=0$};
		\draw (9,0)--(9,6);
		\draw [below] (9,0) node {$x=\infty$};
		\draw (0,1)--(10,1);
		\draw [right] (10,1) node {$T_{x}=0$};
		\draw (0,5)--(10,5);
		\draw [right] (10,5) node {$T_{x}=\infty$};
		\filldraw (1,3) circle [radius=0.05];
		\draw [left] (1,3) node {$1,q$};
		\draw (1,3) circle [radius=0.1];
		\filldraw (9,3) circle [radius=0.05];
		\draw (9,3) circle [radius=0.1];
		\draw [right] (9,3) node {$B/A,Bq^{-1}/A$};
		\filldraw (2,1) circle [radius=0.05];
		\draw [below] (2,1) node{$a_{1}/B$};
		\filldraw (5,1) circle [radius=0.05];
		\draw [below] (5,1) node{$a_{2}/B$};
		\filldraw (8,1) circle [radius=0.05];
		\draw [below] (8,1) node{$a_{3}/B$};
		\filldraw (2,5) circle [radius=0.05];
		\draw [above] (2,5) node{$b_{1}/A$};
		\filldraw (5,5) circle [radius=0.05];
		\draw [above] (5,5) node{$b_{2}/A$};
		\filldraw (8,5) circle [radius=0.05];
		\draw [above] (8,5) node{$b_{3}/A$};
	\end{tikzpicture}
	\caption{the configuration of the equation (\ref{defE3}).}
	\label{conE3}
\end{figure}
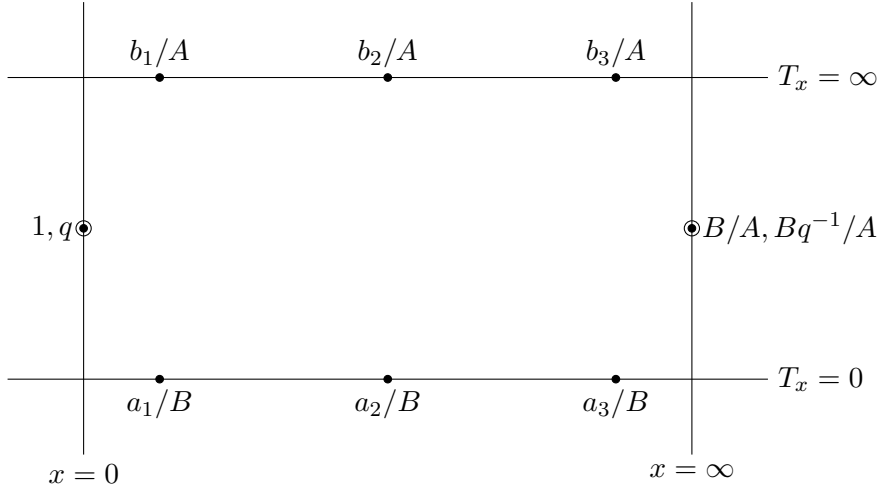
Therefore we identify $\mathcal{E}_{3}f(x)=0$ with the variant of the $q$-hypergeometric equation of degree three {$\mathcal{H}_{3}f(x)=0$} (\ref{defqhyp3}).
{For more details, see the following remark.}
\begin{rem}\label{remE3toH3}
	The change of the dependent variable and parameters, which transforms $\mathcal{E}_{3}f(x)=0$ into $\mathcal{H}_{3}g(x)=0$, is given as follows:
	\begin{align}\label{changepara}
		\notag&g(x)=x^{\nu-\alpha}f(x),\ \ q^{-\nu}=B/A,\ \ q^{l_{i}-1/2}t_{i}=b_{i}/A,\ \ q^{h_{i}+1/2}t_{i}=a_{i}/B, \\
		&\nu=\frac{1}{2}(h_{1}+h_{2}+h_{3}-l_{1}-l_{2}-l_{3}+1)
	\end{align}
\end{rem}
The general solution of $(1-q^{-2}T_{x})(B-Aq^{-1}T_{x})f(x)=0$ is $C_{1}x^{2}+C_{2}x^{\lambda+1}$, where $C_{1}$, $C_{2}$ are pseudo-constants and $B/A=q^{\lambda}$.
If a pseudo-constant $C$ is holomorphic at $x=0$ or $x=\infty$, then $C$ is a constant.
Thus we have 
\begin{align}\label{C1C2}
	\mathcal{E}_{3}\varphi(x,\tau)=C_{1}x^{2}+C_{2}x^{\lambda+1},
\end{align} 
where $C_{1}$, $C_{2}$ are constants.
Therefore we should calculate these constants $C_{1}$, $C_{2}$.
Lemma \ref{lem3psi3} is useful for calculating them.
\begin{lemm}\label{lem3psi3}
	We have
	\begin{align}\label{lim3psi3}
		\lim_{z\to1}(1-z){}_{3}\psi_{3}\left(\begin{array}{c}
			a_{1},a_{2},a_{3}\\
			b_{1},b_{2},b_{3}
		\end{array};z\right)=\frac{(a_{1},a_{2},a_{3})_{\infty}}{(b_{1},b_{2},b_{3})_{\infty}}.
	\end{align}
	Here, ${}_{3}\psi_{3}$ is the bilateral $q$-hypergeometric function
	\begin{align}
		{}_{3}\psi_{3}\left(\begin{array}{c}
			a_{1},a_{2},a_{3}\\
			b_{1},b_{2},b_{3}
		\end{array};z\right)=\sum_{n\in\mathbb{Z}}\frac{(a_{1},a_{2},a_{3})_{n}}{(b_{1},b_{2},b_{3})_{n}}z^{n}.
	\end{align}
\end{lemm}
\begin{proof}
	The negative power part of the function ${}_{3}\psi_{3}$ is holomorphic at $z=1$, thus we have
	\begin{align}
		\lim_{z\to1}(1-z){}_{3}\psi_{3}\left(\begin{array}{c}
			a_{1},a_{2},a_{3}\\
			b_{1},b_{2},b_{3}
		\end{array};z\right)=\lim_{z\to1}(1-z)\sum_{n=0}^{\infty}\frac{(a_{1},a_{2},a_{3})_{n}}{(b_{1},b_{2},b_{3})_{n}}z^{n}.
	\end{align}
	Using the $q$-binomial theorem, we obtain
	\begin{align}
		\notag&\sum_{n=0}^{\infty}\frac{(a_{1},a_{2},a_{3})_{n}}{(b_{1},b_{2},b_{3})_{n}}z^{n}\\
		\notag=&\frac{(a_{1},a_{2},a_{3})_{\infty}}{(b_{1},b_{2},b_{3})_{\infty}}\sum_{n=0}^{\infty}\frac{(b_{1}q^{n},b_{2}q^{n},b_{3}q^{n})_{\infty}}{(a_{1}q^{n},a_{2}q^{n},a_{3}q^{n})_{\infty}}z^{n}\\
		\notag=&\frac{(a_{1},a_{2},a_{3})_{\infty}}{(b_{1},b_{2},b_{3})_{\infty}}\sum_{n=0}^{\infty}\sum_{k_{1}=0}^{\infty}\sum_{k_{2}=0}^{\infty}\sum_{k_{3}=0}^{\infty}\frac{(b_{1}/a_{1})_{k_{1}}(b_{2}/a_{2})_{k_{2}}(b_{3}/a_{3})_{k_{3}}}{(q)_{k_{1}}(q)_{k_{2}}(q)_{k_{3}}}a_{1}^{k_{1}}a_{2}^{k_{2}}a_{3}^{k_{3}}(zq^{k_{1}+k_{2}+k_{3}})^{n}\\
		\notag=&\frac{(a_{1},a_{2},a_{3})_{\infty}}{(b_{1},b_{2},b_{3})_{\infty}}\sum_{k_{1}=0}^{\infty}\sum_{k_{2}=0}^{\infty}\sum_{k_{3}=0}^{\infty}\frac{(b_{1}/a_{1})_{k_{1}}(b_{2}/a_{2})_{k_{2}}(b_{3}/a_{3})_{k_{3}}}{(q)_{k_{1}}(q)_{k_{2}}(q)_{k_{3}}}a_{1}^{k_{1}}a_{2}^{k_{2}}a_{3}^{k_{3}}\frac{1}{1-zq^{k_{1}+k_{2}+k_{3}}}\\
		=&\frac{(a_{1},a_{2},a_{3})_{\infty}}{(b_{1},b_{2},b_{3})_{\infty}}\left(\frac{1}{1-z}+\sum_{k_{1}+k_{2}+k_{3}>0}\frac{(b_{1}/a_{1})_{k_{1}}(b_{2}/a_{2})_{k_{2}}(b_{3}/a_{3})_{k_{3}}}{(q)_{k_{1}}(q)_{k_{2}}(q)_{k_{3}}}a_{1}^{k_{1}}a_{2}^{k_{2}}a_{3}^{k_{3}}\frac{1}{1-zq^{k_{1}+k_{2}+k_{3}}}\right).
	\end{align}
	Therefore we get the desired equation (\ref{lim3psi3}).
\end{proof}
\begin{lemm}\label{lemintcalcu}
	Suppose $\tau\in\{q/a_{1},q/a_{2},q/a_{3},q/(Ax)\}$ and $B/A\notin q^{\mathbb{Z}}$.
	Then the Jackson integral of Jordan-Pochhammer type (\ref{JP}) satisfies the equation $\mathcal{E}_{3}\varphi=(1-q)q(A-B)x^{2}$.
\end{lemm}
\begin{proof}
	For $\tau\in\{q/a_{1},q/a_{2},q/a_{3}\}$, the integral (\ref{JP}) becomes
	\begin{align}
		\varphi(x,\tau)=&\int_{0}^{\tau}t\frac{(Axt)_{\infty}}{(Bxt)_{\infty}}\frac{(a_{1}t,a_{2}t,a_{3}t)_{\infty}}{(b_{1}t,b_{2}t,b_{3}t)_{\infty}}\frac{d_{q}t}{t}\notag\\
		=&(1-q)\sum_{n=0}^{\infty}\frac{(Ax\tau q^{n})_{\infty}}{(Bx\tau q^{n})_{\infty}}\frac{(a_{1}\tau q^{n},a_{2}\tau q^{n},a_{3}\tau q^{n})_{\infty}}{(b_{1}\tau q^{n},b_{2}\tau q^{n},b_{3}\tau q^{n})_{\infty}}\tau q^{n}.
	\end{align}
	Thus we can apply the $q$-binomial theorem to $\displaystyle\frac{(Ax\tau q^{n})_{\infty}}{(Bx\tau q^{n})_{\infty}}$ if $|Bx\tau|<1$.
	Using the $q$-binomial theorem, we have
	\begin{align}
		\varphi(x,\tau)=&(1-q)\sum_{n=0}^{\infty}\sum_{m=0}^{\infty}\frac{(A/B)_{m}}{(q)_{m}}(Bx\tau q^{n})^{m}\frac{(a_{1}\tau q^{n},a_{2}\tau q^{n},a_{3}\tau q^{n})_{\infty}}{(b_{1}\tau q^{n},b_{2}\tau q^{n},b_{3}\tau q^{n})_{\infty}}\tau q^{n}.
	\end{align}
	Hence $\mathcal{E}_{3}\varphi(x,\tau)$ can be written as $\mathcal{E}_{3}\varphi(x,\tau)=\sum_{m=0}^{\infty}\alpha_{m}x^{m}$.
	By (\ref{C1C2}) and the condition $Bq/A\notin q^{\mathbb{Z}}$, we obtain $\mathcal{E}_{3}\varphi(x,\tau)=\alpha_{2}x^{2}$.
	By direct calculations, we have
	\begin{align}
		\alpha_{2}&=(1-q)(B-A)q\frac{(a_{1}\tau,a_{2}\tau,a_{3}\tau)_{\infty}}{(b_{1}\tau,b_{2}\tau,b_{3}\tau)_{\infty}}\notag\\
		&\times\sum_{k=1}^{3}\left(\left(-\frac{\tau}{q}\right)^{k}e_{k}(a)-(-\tau)^{k}e_{k}(b)\right){}_{3}\psi_{3}\left(\begin{array}{c}
			b_{1}\tau,b_{2}\tau,b_{3}\tau\\
			a_{1}\tau,a_{2}\tau,a_{3}\tau
		\end{array};q^{k}\right).
	\end{align}
	The function ${}_{3}\psi_{3}\left(\begin{array}{c}
		b_{1}\tau,b_{2}\tau,b_{3}\tau\\
		a_{1}\tau,a_{2}\tau,a_{3}\tau
	\end{array};z\right)$ satisfies the $q$-difference equation
	\begin{align}
		[&(1-a_{1}\tau q^{-1}T_{z})(1-a_{2}\tau q^{-1}T_{z})(1-a_{3}\tau q^{-1}T_{z})\notag\\
		&-z(1-b_{1}\tau T_{z})(1-b_{2}\tau T_{z})(1-b_{3}\tau T_{z})]{}_{3}\psi_{3}=0.
	\end{align}
	This equation can be rewritten as
	\begin{align}
		\sum_{k=0}^{3}\left(\left(-\frac{\tau}{q}\right)^{k}e_{k}(a)-(-\tau)^{k}e_{k}(b)\right)T_{z}^{k}{}_{3}\psi_{3}=0.
	\end{align}
	Therefore we have
	\begin{align}
		&\sum_{k=1}^{3}\left(\left(-\frac{\tau}{q}\right)^{k}e_{k}(a)-(-\tau)^{k}e_{k}(b)\right){}_{3}\psi_{3}\left(\begin{array}{c}
			b_{1}\tau,b_{2}\tau,b_{3}\tau\\
			a_{1}\tau,a_{2}\tau,a_{3}\tau
		\end{array};q^{k}\right)\notag\\
		&=-\lim_{z\to1}(1-z){}_{3}\psi_{3}\left(\begin{array}{c}
			b_{1}\tau,b_{2}\tau,b_{3}\tau\\
			a_{1}\tau,a_{2}\tau,a_{3}\tau
		\end{array};z\right).		
	\end{align}
	Using Lemma \ref{lem3psi3}, we get $\alpha_{2}=(1-q)q(A-B)$.
	
	Next, we consider the case $\tau=q/(Ax)$.
	By changing the variable $t\to qt/A$, the integral (\ref{JP}) is rewritten as
	\begin{align}
		\varphi(x,\tau)=\frac{q}{Ax}\int_{0}^{1}t\frac{(qt)_{\infty}}{(qBt/A)_{\infty}}\frac{(qa_{1}t/(Ax),qa_{2}t/(Ax),qa_{3}t/(Ax))_{\infty}}{(qb_{1}t/(Ax),qb_{2}t/(Ax),qb_{3}t/(Ax))_{\infty}}\frac{d_{q}t}{t}.
	\end{align}
	Applying the $q$-binomial theorem to $\displaystyle \frac{(qa_{i}t/(Ax))_{\infty}}{(qb_{i}t/(Ax))_{\infty}}$, we get
	\begin{align}\label{Axint}
		&\varphi(x,\tau)\notag\\
		&=\frac{q}{Ax}\sum_{n_{1},n_{2},n_{3}=0}^{\infty}\int_{0}^{1}t^{1+n_{1}+n_{2}+n_{3}}\frac{(qt)_{\infty}}{(qBt/A)_{\infty}}\prod_{i=1}^{3}\left(\frac{(a_{i}/b_{i})_{n_{i}}}{(q)_{n_{i}}}\left(\frac{qb_{i}}{A}\right)^{n_{i}}\right)\frac{d_{q}t}{t}x^{-n_{1}-n_{2}-n_{3}}.
	\end{align}
	Thus we find that $\mathcal{E}_{3}\varphi(x,\tau)=\sum_{m=0}^{\infty}\beta_{m}x^{2-m}$.
	Similar to the case $\tau\in\{q/a_{1},q/a_{2},q/a_{3}\}$, we have $\mathcal{E}_{3}\varphi(x,\tau)=\beta_{0}x^{2}$.
	From (\ref{defE3}) and (\ref{Axint}), we get
	\begin{align}
		\beta_{0}=(B-Aq^{-1})(B-A)q\frac{q}{A}\int_{0}^{1}t\frac{(qt)_{\infty}}{(qBt/A)_{\infty}}\frac{d_{q}t}{t}=q(1-q)(A-B).
	\end{align}
	This completes the proof.
\end{proof}
We can consider the above discussion when replacing the integrand $\psi$ with $\psi C$, where $C$ is {some} pseudo-constant.
The function
\begin{align}
	C(t)=t^{\alpha-\beta}\frac{\theta(q^{\alpha}t)}{\theta(q^{\beta}t)},
\end{align}
is a pseudo-constant, that is, $T_{t}C(t)=C(t)$.
Here, $\theta(t)=(t,q/t)_{\infty}$.
Thus the function
\begin{align}
	\tilde{C}(x,t)=x^{\lambda}t^{-2}\frac{\theta(Bxt)}{\theta(Axt)}\frac{\theta(b_{1}t)}{\theta(a_{1}t)}\frac{\theta(b_{2}t)}{\theta(a_{2}t)}\frac{\theta(b_{3}t)}{\theta(a_{3}t)}
\end{align}
is a pseudo-constant for $x$ and $t$.
Therefore we find that the integral
\begin{align}
	&\label{definttilde}\tilde{\varphi}(x,\sigma)=\int_{0}^{\sigma^{-1}\infty}\psi(x,t)\times \tilde{C}(x,t)\frac{d_{q}t}{t}
	=x^{\lambda}\int_{0}^{\sigma\infty}\tilde{\psi}(x,s)\frac{d_{q}s}{s},\\
	&\tilde{\psi}(x,s)=s\frac{(qs/(Bx))_{\infty}}{(qs/(Ax))_{\infty}}\frac{(qs/b_{1},qs/b_{2},qs/b_{3})_{\infty}}{(qs/a_{1},qs/a_{2},qs/a_{3})_{\infty}},
\end{align}
also satisfies the equation (\ref{eqJP3}).
The following lemma can be proved in the same way as for Lemma \ref{lemintcalcu}, so we do not prove it here.
\begin{lemm}\label{lemintcalcu2}
	Suppose $\sigma\in\{b_{1},b_{2},b_{3},Bx\}$ and $B/A\notin q^{\mathbb{Z}}$. Then the integral (\ref{definttilde}) satisfies the equation $\mathcal{E}_{3}\tilde{\varphi}(x,\sigma)=(1-q)q^{-1}(B-A)x^{\lambda+1}$.
\end{lemm}
By taking the difference between $\varphi(x,\tau_{1})$ and $\varphi(x,\tau_{2})$, or between $\tilde{\varphi}(x,\sigma_{1})$ and $\tilde{\varphi}(x,\sigma_{2})$, we obtain integral solutions for the equation $\mathcal{E}_{3}f(x)=0$.
\begin{thm}\label{thmint3}
	Let $\tau_{1}$, $\tau_{2}\in\{q/a_{1},q/a_{2},q/a_{3},q/(Ax)\}$ and $\sigma_{1}$, $\sigma_{2}\in\{b_{1},b_{2},b_{3},Bx\}$.
	We suppose $q^{\lambda}=B/A\notin q^{\mathbb{Z}}$ and $Aa_{1}a_{2}a_{3}=q^{2}Bb_{1}b_{2}b_{3}$.
	Then the integrals
	\begin{align}
		\label{phi3}&\varphi_{3}(x,\tau_{1},\tau_{2})=\int_{\tau_{1}}^{\tau_{2}}\frac{(Axt,a_{1}t,a_{2}t,a_{3}t)_{\infty}}{(Bxt,b_{1}t,b_{2}t,b_{3}t)_{\infty}}d_{q}t,\\
		\label{tildephi3}&\tilde{\varphi}_{3}(x,\sigma_{1},\sigma_{2})=x^{\lambda}\int_{\sigma_{1}}^{\sigma_{2}}\frac{(qs/(Bx),qs/b_{1},qs/b_{2},qs/b_{3})_{\infty}}{(qs/(Ax),qs/a_{1},qs/a_{2},qs/a_{3})_{\infty}}d_{q}s,
	\end{align}
	satisfy the equation $\mathcal{E}_{3}f(x)=0$ {(\ref{defE3})}.
\end{thm}
\begin{rem}
	By the $q$-binomial theorem, these integrals are $q$-analogs of the integral (\ref{RiePapint}), more precisely, by taking the limit $q\to1$, we have
	\begin{align}
		&\int_{\tau_{1}}^{\tau_{2}}\frac{(Axt,a_{1}t,a_{2}t,a_{3}t)_{\infty}}{(Bxt,b_{1}t,b_{2}t,b_{3}t)_{\infty}}d_{q}t\notag\\
		&\to\int_{\tau_{1}'}^{\tau_{2}'}(1-tx)^{-\nu}(1-tt_{1})^{\nu+l_{1}-h_{1}-1}(1-tt_{2})^{\nu+l_{2}-h_{2}-1}(1-tt_{3})^{\nu+l_{3}-h_{3}-1}dt,
	\end{align}
	with the change of parameters (\ref{changepara}).
	Here, $\tau_{i}\to\tau_{i}'$.
\end{rem}
By considering the case $M=2$ in Proposition \ref{propeqJP}, we obtain the integral solutions for the variant of the $q$-hypergeometric equation of degree two.
In the same way as for the case $M=3$, we have the following theorem.
\begin{thm}\label{thmint2}
	Let $\tau_{1}$, $\tau_{2}\in\{0,q/a_{1},q/a_{2},q/(Ax)\}$ and $\sigma_{1},\sigma_{2}\in\{b_{1},b_{2},Bx,\sigma\infty\}$, where, $\sigma$ is an arbitrary constant.
	We suppose $q^{\lambda}=B/A\notin q^{\mathbb{Z}_{\geq-1}}$, $q^{\alpha+1}B/A\notin q^{\mathbb{Z}_{\leq 0}}$ and $Aa_{1}a_{2}=q^{\alpha+1}Bb_{1}b_{2}$.
	Then the integrals
	\begin{align}
		\label{phi2}&\varphi_{2}(x,\tau_{1},\tau_{2})=\int_{\tau_{1}}^{\tau_{2}}t^{\alpha}\frac{(Axt,a_{1}t,a_{2}t)_{\infty}}{(Bxt,b_{1}t,b_{2}t)_{\infty}}\frac{d_{q}t}{t},\\
		\label{tildephi2}&\tilde{\varphi}_{2}(x,\sigma_{1},\sigma_{2})=x^{\lambda}\int_{\sigma_{1}}^{\sigma_{2}}\frac{(qs/(Bx),qs/b_{1},qs/b_{2})_{\infty}}{(qs/(Ax),qs/a_{1},qs/a_{2})_{\infty}}d_{q}s,
	\end{align}
	satisfy the equation $\mathcal{E}_{2}f(x)=0$, where
	\begin{align}
		\mathcal{E}_{2}=[&x^{2}(1-q^{\alpha}T_{x})(B-AT_{x})-x(e_{1}(a)-q^{\alpha}e_{1}(b)T_{x})(1-T_{x})\notag\\
		&+e_{2}(a)B^{-1}(1-q^{-1}T_{x})(1-T_{x})]T_{x}^{-1}.
	\end{align}
\end{thm}
\begin{rem}\label{remE2toH2}
	The configuration of the equation $\mathcal{E}_{2}f(x)=0$ is given as follows.
	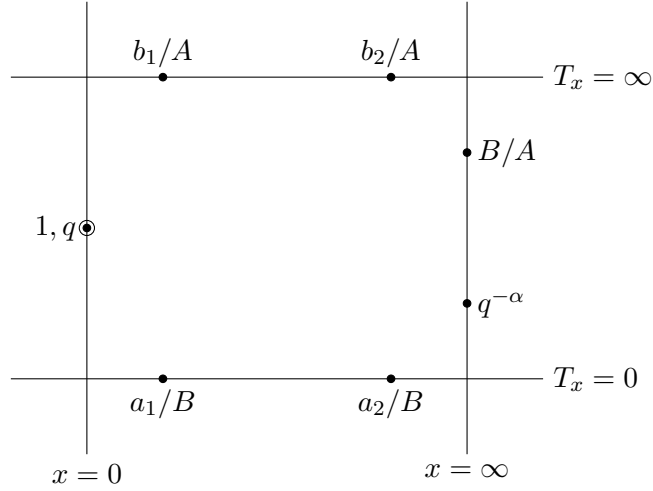
\begin{figure}[H]
		\centering
		\begin{tikzpicture}
			\draw (1,0)--(1,6);
			\draw [below] (1,0) node {$x=0$};
			\draw (6,0)--(6,6);
			\draw [below] (6,0) node {$x=\infty$};
			\draw (0,1)--(7,1);
			\draw [right] (7,1) node {$T_{x}=0$};
			\draw (0,5)--(7,5);
			\draw [right] (7,5) node {$T_{x}=\infty$};
			\filldraw (1,3) circle [radius=0.05];
			\draw [left] (1,3) node {$1,q$};
			\draw (1,3) circle [radius=0.1];
			\filldraw (6,2) circle [radius=0.05];
			\draw [right] (6,2) node {$q^{-\alpha}$};
			\filldraw (6,4) circle [radius=0.05];
			\draw [right] (6,4) node {$B/A$};
			\filldraw (2,1) circle [radius=0.05];
			\draw [below] (2,1) node{$a_{1}/B$};
			\filldraw (5,1) circle [radius=0.05];
			\draw [below] (5,1) node{$a_{2}/B$};
			\filldraw (2,5) circle [radius=0.05];
			\draw [above] (2,5) node{$b_{1}/A$};
			\filldraw (5,5) circle [radius=0.05];
			\draw [above] (5,5) node{$b_{2}/A$};
		\end{tikzpicture}
		\caption{the configuration of $\mathcal{E}_{2}f(x)=0$.}
		\label{conE2}
	\end{figure}
	Therefore the change of the dependent variable and parameters, which transforms $\mathcal{E}_{2}f(x)=0$ into $\mathcal{H}_{2}g(x)=0$, is given as follows:
	\begin{align}\label{changepara2}
		\notag&g(x)=x^{\lambda_{0}}f(x),\ \ \alpha=\lambda_{0}+\alpha_{1},\ \ B/A=q^{-\alpha_{2}-\lambda_{0}}, \\
		\notag&a_{i}/B=q^{h_{i}+1/2}t_{i},\ \ b_{i}/A=q^{l_{i}-1/2}t_{i}, \\
		&\lambda_{0}=\frac{1}{2}(h_{1}+h_{2}-l_{1}-l_{2}-\alpha_{1}-\alpha_{2}+1.
	\end{align}
	In \cite{HMST}, it was shown that the function
	\begin{align}
		g_{1}(x)=x^{-\alpha_{1}}\Phi^{(1)}\left(\begin{array}{c}
			q^{\lambda_{0}+\alpha_{1}}; q^{\lambda_{0}+\alpha_{1}-h_{2}+l_{2}}, q^{\lambda_{0}+\alpha_{1}-h_{1}+l_{1}}\\
			q^{\alpha_{1}-\alpha_{2}+1}
		\end{array};q^{l_{1}+1/2}\frac{t_{1}}{x}, q^{l_{2}+1/2}\frac{t_{2}}{x}\right),
	\end{align}
	satisfies the equation $\mathcal{H}_{2}g_{1}(x)=0$, where $\Phi^{(1)}$ is the $q$-Appell hypergeometric series given by
	\begin{align}
		\Phi^{(1)}\left(\begin{array}{c}
			a; b_{1},b_{2}\\
			c
		\end{array};x_{1},x_{2}\right)=\sum_{n_{1}=0}^{\infty}\sum_{n_{2}=0}^{\infty}\frac{(a)_{n_{1}+n_{2}}(b_{1})_{n_{1}}(b_{2})_{n_{2}}}{(c)_{n_{1}+n_{2}}(q)_{n_{1}}(q)_{n_{2}}}x_{1}^{n_{1}}x_{2}^{n_{2}}.
	\end{align}
	By using Andrews's formula \cite{And}
	\begin{align}
		\int_{0}^{1}t^{\alpha}\frac{(qt,b_{1}x_{1}t,b_{2}x_{2}t)_{\infty}}{(ct/q^{\alpha},x_{1}t,x_{2}t)_{\infty}}\frac{d_{q}t}{t}=(1-q)\frac{(q,c)_{\infty}}{(q^{\alpha},c/q^{\alpha})_{\infty}}\Phi^{(1)}\left(\begin{array}{c}
			q^{\alpha}; b_{1},b_{2}\\
			c
		\end{array};x_{1},x_{2}\right),
	\end{align} 
	the solution $g_{1}(x)$ can be obtained from Theorem \ref{thmint2}.
	More precisely, we have
	\begin{align}
		g_{1}(x)=\frac{1}{1-q}\frac{(q^{\lambda_{0}+\alpha_{1}},q^{-\alpha_{2}+1-\lambda_{0}})_{\infty}}{(q,q^{\alpha_{1}-\alpha_{2}+1})_{\infty}}\left(\frac{A}{q}\right)^{\lambda_{0}+\alpha_{1}}x^{\lambda_{0}}\varphi_{2}(x,0,q/(Ax)),
	\end{align}
	with the change of parameters (\ref{changepara2}).
\end{rem}
\begin{rem}\label{remE3toE2}
	We can also derive Theorem \ref{thmint2} by taking a limit in Theorem \ref{thmint3}.
	We put $b_{3}=q^{\alpha-1}a_{3}$ and consider the limit $a_{3}\to\infty$.
	In the same way as for Remark \ref{Remdegeeq}, the equation $\mathcal{E}_{3}f(x)=0$ becomes $\mathcal{E}_{2}f(x)=0$ by this limit.
	Also we find easily that the integrand of the integral $\tilde{\varphi}_{3}$ (\ref{tildephi3}) becomes the integrand of $\tilde{\varphi}_{2}$ (\ref{tildephi2}) by the same limit.
	By multiplying the integrand of $\tilde{\varphi}_{2}$ (\ref{tildephi2}) by some pseudo-constant, we get the integral (\ref{phi2}).
\end{rem}
We consider the linear independence of the integral solutions.
If functions $f(x)$ and $g(x)$ are linear dependent over the field of pseudo-constants $K=\{C(x)\mid C(qx)=C(x)\}$ for general parameters, then $f(x)$ and $g(x)$ with special parameters are linearly dependent too.
Therefore it is enough to consider linearly independence for the integrals in some special case.
\begin{prop}\label{lemlin}
	Suppose $a_{i}\neq a_{j}$, $b_{i}\neq b_{j}$ $(i\neq j)$.
	\begin{itemize}
		\item[(1)] Let $\tau_{i}\in\{q/a_{1},q/a_{2},q/a_{3},q/(Ax)\}$ for $i=1,\ldots,4$.
		If $\{\tau_{1},\tau_{2}\}\neq\{\tau_{3},\tau_{4}\}$, then the integrals $\varphi_{3}(x,\tau_{1},\tau_{2})$ and $\varphi_{3}(x,\tau_{3},\tau_{4})$, defined by (\ref{phi3}), are linearly independent over the field of pseudo-constants $K$.
		\item[(2)] Let $\sigma_{i}\in\{b_{1},b_{2},b_{3},Bx\}$ for $i=1,\ldots,4$.
		If $\{\sigma_{1},\sigma_{2}\}\neq \{\sigma_{3},\sigma_{4}\}$, then the integrals $\tilde{\varphi}_{3}(x,\sigma_{1},\sigma_{2})$ and $\tilde{\varphi}_{3}(x,\sigma_{3},\sigma_{4})$, defined by (\ref{tildephi3}), are linearly independent over $K$.
		\item[(3)] Let $a_{0}=Ax$ and $b_{0}=Bx$.
		If $\{i,j\}\neq \{k,l\}$ $(0\leq i,j,k,l\leq 3)$, then the integrals $\varphi_{3}(x,q/a_{i},q/a_{j})$ and $\tilde{\varphi}_{3}(x,b_{k},b_{l})$ are linearly independent over $K$.
	\end{itemize}
\end{prop}
\begin{proof}
	We consider the special case $a_{i}=b_{i}$ for $i=1$, 2, 3.
	Because of the condition $Aa_{1}a_{2}a_{3}=q^{2}Bb_{1}b_{2}b_{3}$,  we have $A=q^{2}B$.
	By a simple calculation, we have
	\begin{align}
		\label{spephi3}&\varphi_{3}(x,\tau_{1},\tau_{2})=\int_{\tau_{1}}^{\tau_{2}}\frac{(q^{2}Bxt)_{\infty}}{(Bxt)_{\infty}}d_{q}t=\frac{(1-q)(\tau_{2}-\tau_{1})}{(1-Bx\tau_{1})(1-Bx\tau_{2})},\\
		\label{spetildephi3}&\tilde{\varphi}_{3}(x,\sigma_{1},\sigma_{2})=x^{-2}\int_{\sigma_{1}}^{\sigma_{2}}\frac{(qs/(Bx))_{\infty}}{(qs/(q^{2}Bx))_{\infty}}d_{q}s=\frac{(1-q)(qB)^{2}(\sigma_{2}-\sigma_{1})}{(qBx-\sigma_{1})(qBx-\sigma_{2})}.
	\end{align} 
	We prove only the case (1).
	
	We assume $C_{1}(x)\varphi_{3}(x,\tau_{1},\tau_{2})+C_{2}(x)\varphi_{3}(x,\tau_{3},\tau_{4})=0$, where $C_{1}(x)$, $C_{2}(x)\in K$.
	By (\ref{spephi3}) and the conditions for $\tau_{i}$, the function $\varphi_{3}(x,\tau_{1},\tau_{2})/\varphi_{3}(x,\tau_{3},\tau_{4})$ is a non-constant rational function.
	Since a non-constant rational function is not a pseudo-constant, we get $C_{1}=C_{2}=0$.

	The cases (2) and (3) are proved by the same way as the case (1).
	More precisely, (2) is proved by considering the ratio of the integrals (\ref{spetildephi3}), and (3) is proved by considering the ratio of the integrals (\ref{spephi3}) and (\ref{spetildephi3}).
\end{proof}
\begin{rem}\label{remintdep}
	In (3) of Proposition \ref{lemlin}, we consider the case $(i,j)=(k,l)$.
	Then the ratio of integrals is a constant.
	Therefore we cannot determine whether the integrals $\varphi_{3}(x,q/a_{i},q/a_{j})$ and $\tilde{\varphi}_{3}(x,b_{i},b_{j})$ are linearly independent by the above method.
\end{rem}

\section{Series solutions}\label{secser}
In this section, we present solutions for the variant of the $q$-hypergeometric solutions of degree three by the very-well-poised-balanced $q$-hypergeometric series ${}_{8}W_{7}$.
In section \ref{secint}, we obtain integral solutions for variants of the $q$-hypergeometric equation.
The ${}_{8}W_{7}$ solutions are obtained by transforming the integral.
In this section, we use some formulas for the $q$-hypergeometric series without proofs.
Those formulas are summarized in \cite{GR}.
\begin{defn}[\cite{GR}, (2.1.11)]
	The very-well-poised $q$-hypergeometric series is defined as follows:
	\begin{align}
		{}_{r+1}W_{r}(a_{1}; a_{4},\ldots, a_{r+1};z)=\sum_{n=0}^{\infty}\frac{1-a_{1}q^{2n}}{1-a_{1}}\frac{(a_{1},a_{4},\ldots,a_{r+1})_{n}}{(q,qa_{1}/a_{4},\ldots,qa_{1}/a_{r+1})_{n}}z^{n}.
	\end{align}
	This series converges for $|z|<1$, and is called balanced if $\displaystyle z=\frac{(\pm (a_{1}q)^{1/2})^{r-3}}{a_{4}a_{5}\cdots a_{r+1}}$.
\end{defn}
\begin{rem}
	The very-well-poised $q$-hypergeometric series ${}_{r+1}W_{r}$ is rewritten as follows:
	\begin{align}
		{}_{r+1}W_{r}(a_{1}; a_{4},\ldots, a_{r+1};z)={}_{r+1}\varphi_{r}\left(\begin{array}{c}
			a_{1},qa_{1}^{1/2},-qa_{1}^{1/2},a_{4},\ldots,a_{r+1}\\
			a_{1}^{1/2},-a_{1}^{1/2},qa_{1}/a_{4},\ldots,qa_{1}/a_{r+1}
		\end{array};z\right).
	\end{align}
\end{rem}
We introduce some properties of the very-well-poised-balanced ${}_{8}W_{7}$ series without proofs.
For proofs and more details, see section 2 of \cite{GR}.
\begin{lemm}[\cite{GR}, (2.10.19)]\label{serlemint}
	If $cd=abefgh$ and $|ah|<1$, then we have
	\begin{align}
		\notag&\int_{a}^{b}\frac{(qt/a,qt/b,ct,dt)_{\infty}}{(et,ft,gt,ht)_{\infty}}d_{q}t\\
		=&b(1-q)\frac{(q,bq/a,a/b,cd/(eh),cd/(fh),cd/(gh),bc,bd)_{\infty}}{(ae,af,ag,be,bf,bg,bh,bcd/h)_{\infty}}
		{}_{8}W_{7}\left(\frac{bcd}{hq};be,bf,bg,\frac{c}{h},\frac{d}{h};ah\right).\label{inttoser}
	\end{align}
\end{lemm}
\begin{lemm}[\cite{GR}, (2.10.1)]\label{serlemtrans}
	Let $\mu=qa^{2}/(bcd)$. We have
	\begin{align}\label{8W7trans}
		{}_{8}W_{7}\left(a;b,c,d,e,f;\frac{a^{2}q^{2}}{bcdef}\right)=\frac{(aq,aq/(ef),\mu q/e,\mu q/f)_{\infty}}{(aq/e,aq/f,\mu q,\mu q/(ef))_{\infty}}{}_{8}W_{7}\left(\mu;\frac{\mu b}{a},\frac{\mu c}{a},\frac{\mu d}{a}, e,f;\frac{aq}{ef}\right),
	\end{align}
	if $\max(|aq/(ef)|, |\mu q/(ef)|)<1$.
\end{lemm}
By applying Lemma \ref{serlemint} to the integral (\ref{phi3}), 
we obtain series solutions for the equation (\ref{defE3}).
\begin{thm}\label{thmser3}
	If $a_{1}a_{2}a_{3}A=q^{2}b_{1}b_{2}b_{3}B$, then the functions
	\begin{align}
		\label{sersol1}&\frac{(Axq/a_{3})_{\infty}}{(Bxq/a_{3})_{\infty}}{}_{8}W_{7}\left(\frac{a_{2}A}{a_{3}B}; \frac{qb_{1}}{a_{3}}, \frac{qb_{2}}{a_{3}}, \frac{qb_{3}}{a_{3}}, \frac{a_{2}}{Bx}, \frac{A}{B}; \frac{qBx}{a_{1}}\right),\\
		\label{sersol1h}&\frac{(a_{3}Ax/(b_{1}b_{3}),a_{3}Ax/(b_{2}b_{3}),qAx/a_{2})_{\infty}}{(qBx/a_{1},qBx/a_{2},qa_{3}Ax/(a_{2}a_{3}))_{\infty}}{}_{8}W_{7}\left(\frac{a_{3}Ax}{a_{2}b_{3}};\frac{qBx}{a_{2}},\frac{qb_{1}}{a_{2}},\frac{qb_{2}}{a_{2}},\frac{a_{3}}{b_{3}},\frac{Ax}{b_{3}};\frac{qb_{3}}{a_{1}}\right),\\
		&\frac{1}{x}\frac{(qa_{1}/(Ax),Ax/a_{1},a_{2}a_{3}/(b_{3}Bx),qa_{2}/(Ax),qa_{3}/(Ax))_{\infty}}{(qBx/a_{1},qb_{1}/(Ax),qb_{2}/(Ax),qb_{3}/(Ax),qa_{2}a_{3}/(b_{3}Ax))_{\infty}}\notag\\
		&\label{sersol2}\times{}_{8}W_{7}\left(\frac{a_{2}a_{3}}{b_{3}Ax};\frac{qb_{1}}{Ax},\frac{qB}{A},\frac{qb_{2}}{Ax},\frac{a_{2}}{b_{3}},\frac{a_{3}}{b_{3}};\frac{qb_{3}}{a_{1}}\right),\\
		\notag&\frac{1}{x}\frac{(qa_{1}/(Ax),qa_{2}/(Ax),qa_{3}/(Ax),Ax/a_{1},a_{2}a_{3}/(b_{1}Bx),a_{2}a_{3}/(b_{2}Bx),a_{2}a_{3}/(b_{3}Bx))_{\infty}}{(qb_{1}/(Ax),qb_{2}/(Ax),qb_{3}/(Ax),qa_{2}a_{3}/(ABx^{2}))_{\infty}}\\
		\label{sersol2h}&\times{}_{8}W_{7}\left(\frac{a_{2}a_{3}}{ABx^{2}};\frac{qb_{1}}{Ax},\frac{qb_{2}}{Ax},\frac{qb_{3}}{Ax},\frac{a_{2}}{Bx},\frac{a_{3}}{Bx};\frac{qBx}{a_{1}}\right),\\
		\label{sersol5}&\frac{(qAx/a_{1},a_{1}/(Ax),a_{2}a_{3}/(b_{3}Bx))_{\infty}}{(qb_{1}/(Ax),qb_{2}/(Ax),qBx/a_{1})_{\infty}}{}_{8}W_{7}\left(\frac{a_{2}a_{3}}{a_{1}b_{3}};\frac{qBx}{a_{1}},\frac{qb_{1}}{a_{1}},\frac{qb_{2}}{a_{1}},\frac{a_{2}}{b_{3}},\frac{a_{3}}{b_{3}};\frac{qb_{3}}{Ax}\right),\\
		&\frac{(qAx/a_{1},a_{1}/(Ax),a_{2}a_{3}/(b_{1}Bx),a_{2}a_{3}/(b_{2}Bx))_{\infty}}{(qb_{1}/(Ax),qb_{2}/(Ax),qb_{3}/(Ax),qBx/a_{1},qa_{2}a_{3}/(a_{1}Bx))_{\infty}}\notag\\
		&\label{sersol5h}\times{}_{8}W_{7}\left(\frac{a_{2}a_{3}}{a_{1}Bx};\frac{qb_{1}}{a_{1}},\frac{qb_{2}}{a_{1}},\frac{qb_{3}}{a_{1}},\frac{a_{2}}{Bx},\frac{a_{3}}{Bx};\frac{qB}{A}\right),\
	\end{align}
	satisfy the equation $\mathcal{E}_{3}f(x)=0$ {(\ref{defE3})}.
\end{thm}
\begin{rem}
	The series (\ref{sersol1}) and (\ref{sersol1h}), (\ref{sersol2}) and (\ref{sersol2h}), (\ref{sersol5}) and (\ref{sersol5h}) correspond to the integrals $\varphi_{3}(x,q/a_{1},q/a_{2})$, $\varphi_{3}(x,q/a_{1},q/(Ax))$, $\varphi_{3}(x,q/(Ax),q/a_{1})$, respectively.
	The right-hand-side of the equation (\ref{inttoser}) in Lemma \ref{serlemint} is symmetric for $e$, $f$, $g$ and also for $c$, $d$.
	Thus the series representations for the integrals changes depending essentially only on the choices of $a$, $b$, $h$.
	In this case, we have two or four representations for each integral.
	
	The solutions that correspond to the integral $\tilde{\varphi}_{3}$ are obtained by acting with $s_{2}$, defined in (\ref{transs2}) below, to the solutions in Theorem \ref{thmser3}.
\end{rem}
\begin{rem}\label{remW87}
	The very-well-poised-balanced $q$-hypergeometric series ${}_{8}W_{7}$ is called the Askey-Wilson function \cite{KS}.
	{Well known linear $q$-difference equations satisfied by the very-well-poised-balanced ${}_{8}W_{7}$ are the ones considered by Ismail and Rahman \cite{IR}.
	We put $\displaystyle\phi={}_{8}W_{7}\left(a;b,c,d,e,f;\frac{(aq)^{2}}{bcdef}\right)$, and $T_{1}=T_{b}T_{c}^{-1}$, $T_{2}=T_{a}^{2}T_{c}T_{d}T_{e}T_{f}$.
	In \cite{IR}, linear relations for $T_{1}\phi$, $\phi$, $T_{1}^{-1}\phi$, and for $T_{2}\phi$, $\phi$, $T_{2}^{-1}\phi$ were obtained.
	These equations are as follows:
	\begin{align}
		&\left[(b-c)\left(1-\frac{aq}{bc}\right)\left(1-\frac{a^{2}q^{2}}{bcdef}\right)+L+M\right]\phi\notag\\
		&=\frac{(1-c)(1-a/c)}{(1-b/q)(1-aq/b)}T_{1}^{-1}\phi+\frac{(1-b)(1-a/b)}{(1-c/q)(1-aq/c)}T_{1}\phi,\\
		\notag&\Biggl[\frac{b(1-a/(bc))(1-a/(bd))(1-a/(be))(1-a/(bf))}{(1-a/b)(1-a/(bq))}\\
		\notag&+\frac{a^{2}(1-c)(1-d)(1-e)(1-f)q}{(1-a/b)(1-aq/b)bcdef}+(1-b)\left(1-\frac{a^{2}q}{bcdef}\right)\Biggr]\phi\\
		\notag&=\frac{(1-a/c)(1-a/d)(1-a/e)(1-a/f)}{(1-a)(1-aq)}T_{2}^{-1}\phi\\
		&+ \frac{a^{2}q(1-c)(1-d)(1-e)(1-f)}{cdef}\notag\\
		&\times\frac{(aq)_{2}(1-aq/(bc))(1-aq/(bd))(1-aq/(be))(1-aq/(bf))}{(a/b)_{2}(aq/b)_{2}(1-aq/c)(1-aq/d)(1-aq/e)(1-aq/f)}T_{2}\phi,
	\end{align}
	where
	\begin{align}
		&L=\frac{c(1-c/q)(1-aq/(cd))(1-aq/(ce))(1-aq/(cf))}{b-c/q},\\ &M=\frac{b(1-b/q)(1-aq/(bd))(1-aq/(be))(1-aq/(bf))}{c-b/q}.
	\end{align}
	}On the other hand, we find that the function ${}_{8}W_{7}$ satisfies  the equation $\mathcal{H}_{3}f(x)=0$ under some gauge transformation.
	Our result gives linear relations for $T_{3}\phi$, $\phi$, $T_{3}^{-1}\phi$, and for $T_{4}\phi$, $\phi$, $T_{4}^{-1}\phi$, and $T_{5}\phi$, $\phi$, $T_{5}^{-1}\phi$.
	Here, $T_{3}=T_{b}$, $T_{4}=T_{a}T_{b}T_{c}$, and $T_{5}=T_{a}^{2}T_{b}T_{c}T_{d}T_{e}T_{f}$.
	Therefore we find that the equation $\mathcal{H}_{3}f(x)=0$ can be regarded as a contiguity relation for the Askey-Wilson function ${}_{8}W_{7}$.
	We note that the relation for  $T_{3}\phi$, $\phi$, $T_{3}^{-1}\phi$ is discussed in \cite{GM}
	
	In \cite{HMST}, it was mentioned that {some} solution for the equation $\mathcal{H}_{2}f(x)=0$ is related to the big $q$-Jacobi polynomial.
	{The big $q$-Jacobi polynomial can be obtained from the Askey-Wilson polynomial by taking some limit.
	In addition, the little $q$-Jacobi polynomial, which is related to Heine's $q$-hypergeometric equation, can be obtained from the big $q$-Jacobi polynomial.
	Degenerations of orthogonal polynomials are summarized in the Askey scheme \cite{KLS}.
	We find that the degeneration $\mathcal{H}_{3}\to\mathcal{H}_{2}\to\mathcal{H}_{1}$ is related to (a part of) the Askey scheme.}
\end{rem}
\begin{rem}
	Taking the limit $q\to1$, we have
	\begin{align}
		&\frac{(q^{\alpha}X)_{\infty}}{(q^{\beta}X)_{\infty}}\to(1-X)^{\beta-\alpha},\ \ (q^{\alpha}X)_{n}\to(1-X)^{n},\\
		&\frac{(q^{\alpha})_{n}}{(1-q)^{n}}\to\alpha(\alpha+1)\cdots(\alpha+n-1).
	\end{align}
	Thus, by taking $q\to1$ with the change of parameters (\ref{changepara}), some of the solutions in Theorem \ref{thmser3} formally become the Gauss hypergeometric function:
	\begin{align}
		x^{\nu-\alpha}(x-t_{1})^{\mu_{1}}(x-t_{2})^{\mu_{2}}(x-t_{3})^{\mu_{3}}{}_{2}F_{1}\left(\begin{array}{c}
			\nu_{1},\nu_{2}\\
			\nu_{3}
		\end{array};\frac{x-t_{i}}{x-t_{j}}\frac{t_{k}-t_{j}}{t_{k}-t_{i}}\right).
	\end{align}
	Here, $\mu_{i}$, $\nu_{j}$ are suitable parameters determined by $\nu$, $\alpha$, $l_{i}$, $h_{i}$.
	In this sense, some of the solutions in Theorem \ref{thmser3} are $q$-analogs of series solutions of Riemann-Papperitz equation {(\ref{RiePapser})}.
\end{rem}
The equation $\mathcal{E}_{3}f(x)=0$ has some symmetries.
We {recall} the configuration of $\mathcal{E}_{3}f(x)=0$, that is
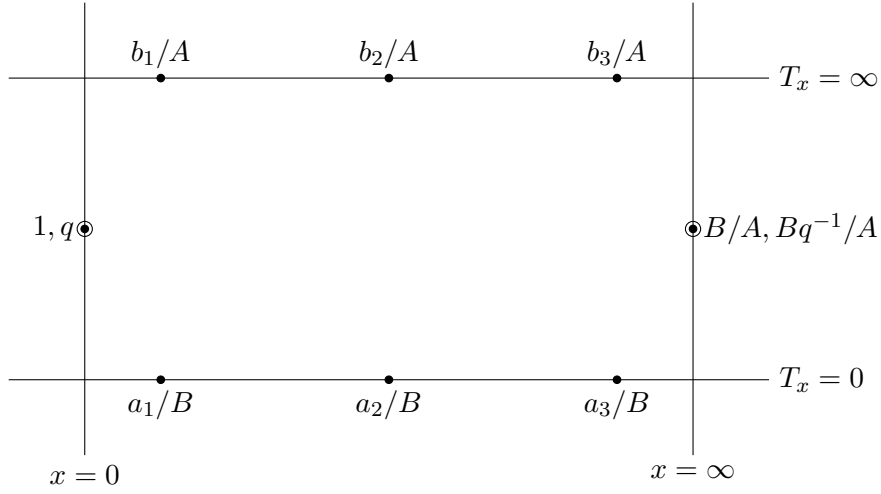
\begin{center}
	\begin{tikzpicture}
		\draw (1,0)--(1,6);
		\draw [below] (1,0) node {$x=0$};
		\draw (9,0)--(9,6);
		\draw [below] (9,0) node {$x=\infty$};
		\draw (0,1)--(10,1);
		\draw [right] (10,1) node {$T_{x}=0$};
		\draw (0,5)--(10,5);
		\draw [right] (10,5) node {$T_{x}=\infty$};
		\filldraw (1,3) circle [radius=0.05];
		\draw [left] (1,3) node {$1,q$};
		\draw (1,3) circle [radius=0.1];
		\filldraw (9,3) circle [radius=0.05];
		\draw (9,3) circle [radius=0.1];
		\draw [right] (9,3) node {$B/A,Bq^{-1}/A$};
		\filldraw (2,1) circle [radius=0.05];
		\draw [below] (2,1) node{$a_{1}/B$};
		\filldraw (5,1) circle [radius=0.05];
		\draw [below] (5,1) node{$a_{2}/B$};
		\filldraw (8,1) circle [radius=0.05];
		\draw [below] (8,1) node{$a_{3}/B$};
		\filldraw (2,5) circle [radius=0.05];
		\draw [above] (2,5) node{$b_{1}/A$};
		\filldraw (5,5) circle [radius=0.05];
		\draw [above] (5,5) node{$b_{2}/A$};
		\filldraw (8,5) circle [radius=0.05];
		\draw [above] (8,5) node{$b_{3}/A$};
	\end{tikzpicture}.
\end{center}
By the gauge transformation $\displaystyle g(x)=\frac{(qBx/a_{1})_{\infty}}{(Ax/b_{1})_{\infty}}f(x)$, this configuration is transformed to the following configuration:
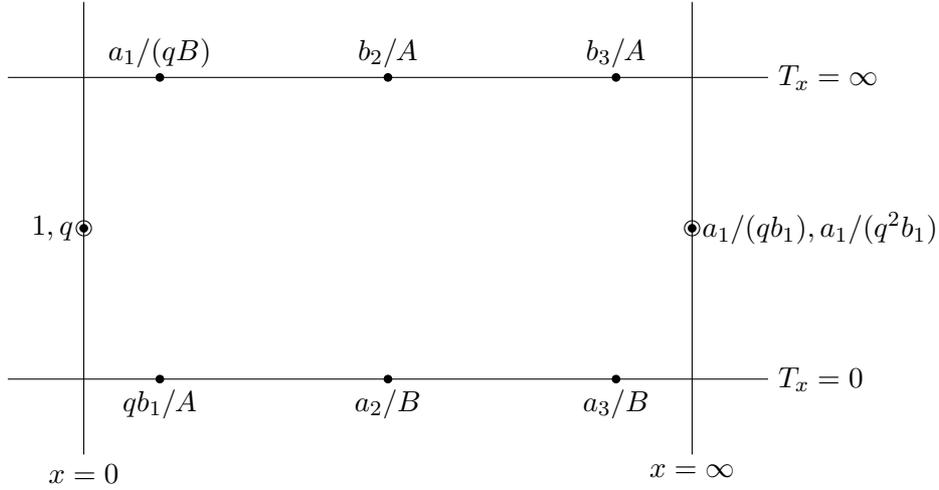
\begin{center}
	\begin{tikzpicture}
		\draw (1,0)--(1,6);
		\draw [below] (1,0) node {$x=0$};
		\draw (9,0)--(9,6);
		\draw [below] (9,0) node {$x=\infty$};
		\draw (0,1)--(10,1);
		\draw [right] (10,1) node {$T_{x}=0$};
		\draw (0,5)--(10,5);
		\draw [right] (10,5) node {$T_{x}=\infty$};
		\filldraw (1,3) circle [radius=0.05];
		\draw [left] (1,3) node {$1,q$};
		\draw (1,3) circle [radius=0.1];
		\filldraw (9,3) circle [radius=0.05];
		\draw (9,3) circle [radius=0.1];
		\draw [right] (9,3) node {$a_{1}/(qb_{1}),a_{1}/(q^{2}b_{1})$};
		\filldraw (2,1) circle [radius=0.05];
		\draw [below] (2,1) node{$qb_{1}/A$};
		\filldraw (5,1) circle [radius=0.05];
		\draw [below] (5,1) node{$a_{2}/B$};
		\filldraw (8,1) circle [radius=0.05];
		\draw [below] (8,1) node{$a_{3}/B$};
		\filldraw (2,5) circle [radius=0.05];
		\draw [above] (2,5) node{$a_{1}/(qB)$};
		\filldraw (5,5) circle [radius=0.05];
		\draw [above] (5,5) node{$b_{2}/A$};
		\filldraw (8,5) circle [radius=0.05];
		\draw [above] (8,5) node{$b_{3}/A$};
	\end{tikzpicture}.
\end{center}
Thus, if a function $f(a_{1},a_{2},a_{3},b_{1},b_{2},b_{3},A,B,x)$ is a solution of $\mathcal{E}_{3}f(x)=0$, then the function $\displaystyle \frac{(Ax/b_{1})_{\infty}}{(qBx/a_{1})_{\infty}}f\left(a_{1},a_{2}\frac{a_{1}A}{qb_{1}B},a_3\frac{a_{1}A}{qb_{1}B},a_{1}\frac{A}{qB},b_{2},b_{3},A,\frac{a_{1}}{qb_{1}}A,x\right)$ satisfies the same equation.
Also by the transformation $\{g(x)=x^{-\lambda}f(x),\ z=x^{-1}\}$, then we get the configuration
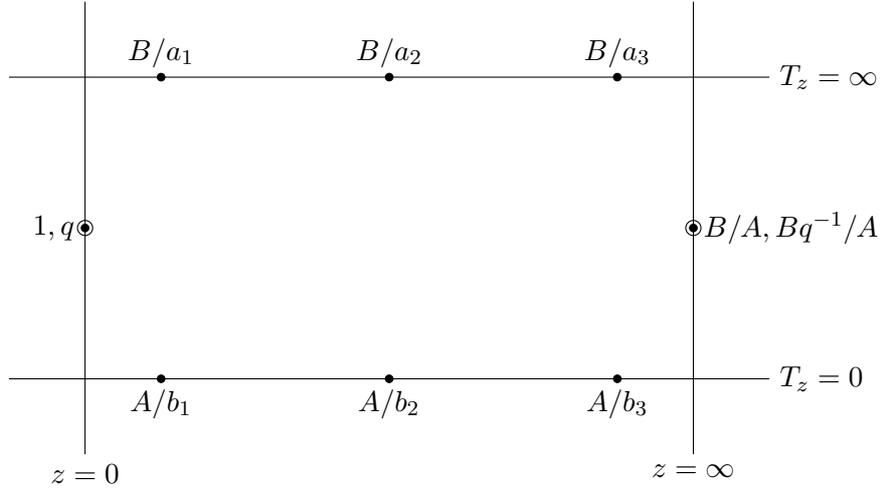
\begin{center}
	\begin{tikzpicture}
		\draw (1,0)--(1,6);
		\draw [below] (1,0) node {$z=0$};
		\draw (9,0)--(9,6);
		\draw [below] (9,0) node {$z=\infty$};
		\draw (0,1)--(10,1);
		\draw [right] (10,1) node {$T_{z}=0$};
		\draw (0,5)--(10,5);
		\draw [right] (10,5) node {$T_{z}=\infty$};
		\filldraw (1,3) circle [radius=0.05];
		\draw [left] (1,3) node {$1,q$};
		\draw (1,3) circle [radius=0.1];
		\filldraw (9,3) circle [radius=0.05];
		\draw (9,3) circle [radius=0.1];
		\draw [right] (9,3) node {$B/A,Bq^{-1}/A$};
		\filldraw (2,1) circle [radius=0.05];
		\draw [below] (2,1) node{$A/b_{1}$};
		\filldraw (5,1) circle [radius=0.05];
		\draw [below] (5,1) node{$A/b_{2}$};
		\filldraw (8,1) circle [radius=0.05];
		\draw [below] (8,1) node{$A/b_{3}$};
		\filldraw (2,5) circle [radius=0.05];
		\draw [above] (2,5) node{$B/a_{1}$};
		\filldraw (5,5) circle [radius=0.05];
		\draw [above] (5,5) node{$B/a_{2}$};
		\filldraw (8,5) circle [radius=0.05];
		\draw [above] (8,5) node{$B/a_{3}$};
	\end{tikzpicture}.
\end{center}
Thus the function $\displaystyle x^{\lambda}f\left(\frac{AB}{b_{1}},\frac{AB}{b_{2}},\frac{AB}{b_{3}},\frac{AB}{a_{1}},\frac{AB}{a_{2}},\frac{AB}{a_{3}},A,B,\frac{1}{x}\right)$ satisfies the same equation.
In addition, the equation $\mathcal{E}_{3}f(x)=0$ is symmetric for $a_{1}$, $a_{2}$, $a_{3}$, and for $b_{1}$, $b_{2}$, $b_{3}$.
Therefore the function $f(a_{i},a_{j},a_{k},b_{i'},b_{j'},b_{k'},A,B,x)$ also satisfies $\mathcal{E}_{3}f(x)=0$, where $(i,j,k)$ and $(i',j',k')$ are permutations of (1, 2, 3).
By acting elements of group $G_{3}=\langle s_{i}\mid 1\leq i\leq 6\rangle$ to the solutions in Theorem \ref{thmser3}, we obtain many solutions.
Here,
\begin{align}
	&s_{1}:(a_{1},a_{2},a_{3},b_{1},b_{2},b_{3},A,B,x)\to\left(a_{1},a_{2}\frac{a_{1}A}{qb_{1}B},a_3\frac{a_{1}A}{qb_{1}B},a_{1}\frac{A}{qB},b_{2},b_{3},A,\frac{a_{1}}{qb_{1}}A,x\right),\\
	\label{transs2}&s_{2}:(a_{1},a_{2},a_{3},b_{1},b_{2},b_{3},A,B,x)\to\left(\frac{AB}{b_{1}},\frac{AB}{b_{2}},\frac{AB}{b_{3}},\frac{AB}{a_{1}},\frac{AB}{a_{2}},\frac{AB}{a_{3}},A,B,\frac{1}{x}\right),\\
	&s_{3}:a_{1}\leftrightarrow a_{2},\ s_{4}:a_{2}\leftrightarrow a_{3},\ s_{5}:b_{1}\leftrightarrow b_{2},\ s_{6}:b_{2}\leftrightarrow b_{3}.
\end{align}
Relations among the $\{s_{i}\}$ are as follows:
\begin{align}
	&s_{i}^{2}=\mathrm{id}\ \ (i=1,\ldots,6),\\
	&(s_{1}s_{2})^{2}=(s_{1}s_{4})^{2}=(s_{1}s_{6})^{2}=(s_{1}s_{3})^{3}=(s_{1}s_{5})^{3}=\mathrm{id},\\
	&(s_{3}s_{5})^{2}=(s_{3}s_{6})^{2}=(s_{4}s_{5})^{2}=(s_{4}s_{6})^{2}=(s_{3}s_{4})^{3}=(s_{5}s_{6})^{3}=\mathrm{id},\\
	&(s_{2}s_{3})^{4}=(s_{2}s_{4})^{4}=(s_{2}s_{5})^{4}=(s_{2}s_{6})^{4}=\mathrm{id},\\
	&s_{2}s_{3}s_{2}s_{5}=s_{2}s_{4}s_{2}s_{6}=\mathrm{id}.
\end{align}
We put 
\begin{align}
\notag&
s
=
s_2 s_3 s_4 s_3 s_6 s_5 s_6 s_1 s_5 s_6 s_3 s_1 s_4 s_3 s_5 s_1: 
(a_{1},a_{2},a_{3},b_{1},b_{2},b_{3},A,B,x)
\\
&
\to 
\left(
\frac{A^3 a_2 a_3}{B b_1 b_2 b_3 q^2}, 
\frac{A^3 a_1 a_3}{B b_1 b_2 b_3 q^2}, 
\frac{A^3 a_1 a_2}{B b_1 b_2 b_3 q^2}, 
\frac{A^2}{b_1 q}, 
\frac{A^2}{b_2 q}, 
\frac{A^2}{b_3 q},
A,
\frac{A^3 a_1 a_2 a_3}{B^2 b_1 b_2 b_3 q^3},
\frac{1}{x}
 \right)
 ,
\end{align} 
then we have $s_i s = s s_i$, $(i = 1, 3, 4, 5, 6)$.
We find $G_{3}=\langle s_{1},s_{3},s_{4},s_{5},s_{6}, s\rangle=\langle s_{1},s_{3},s_{4},s_{5},s_{6}\rangle\times \langle s\rangle\simeq\mathfrak{S}_{6}\times \mathfrak{S}_{2}$.

\section{Summary and discussion}\label{secsum}
In this paper, we obtained two main results for the variant of the $q$-hypergeometric equation of degree three $\mathcal{H}_{3}f(x)=0$.
A summary of these results is as follows.

Main results are Theorem \ref{thmint3} and Theorem \ref{thmser3}, which give the integral solutions and the series solutions, respectively.
These integrals and series are $q$-analogs of the integral (\ref{RiePapint}) and the series (\ref{RiePapser}), respectively.
First we introduced a configuration of $q$-difference equations and characterized the equation $\mathcal{H}_3f(x) = 0$ in section \ref{seceq}.
Second we constructed integral solutions for $\mathcal{H}_3f(x)=0$ by using the above characterization in section \ref{secint}.
Finally we obtained series solutions for $\mathcal{H}_3f(x)=0$ by applying the transformation formula \eqref{inttoser} in section \ref{secser}.
In this way we got two main results.
Because the series solutions are expressed by the very-well-poised-balanced $q$-hypergeometric series ${}_{8}W_{7}$, we found that the variant of the $q$-hypergeometric equation of degree three is regarded as a $q$-difference equation of the Askey-Wilson function with respect to some parameter (see Remark \ref{remW87}).

There are many problems related to our results. We mention four of them here. 
\begin{itemize}
	\item[(1)] In section \ref{secint}, we obtained integral solutions for variants of the $q$-hypergeometric equation.
	In Theorem \ref{thmint3}, we found 12 solutions for the equation $\mathcal{E}_{3}f(x)=0$.
	Since the rank of the equation $\mathcal{E}_{3}f(x)=0$ is 2, there are 10 linear relations among them.
	In the general theory of linear $q$-difference equations (cf. \cite{Bi,Ca}), the connection problem for linear $q$-difference equations is to find the connection matrix $C(x)$ which satisfies $\bm{y}_{\infty}(x)=C(x)\bm{y}_{0}(x)$.
	Here $\bm{y}_{a}$ is a fundamental solution of the $q$-difference equation at $x=a$.
	However, the points $x=0$ and $x=\infty$ are essentially non-singular for the variant of the $q$-hypergeometric equation of degree three.
	We are sure that the linear relations for the integral solutions should be considered  instead of the connection problem between solutions at $x=0$ and $x=\infty$.
	In \cite{FN2024}, linear relations among these integrals are discussed.

	According to Remark \ref{Remdegeeq}, the variant of the $q$-hypergeometric equation of degree three becomes Heine's $q$-hypergeometric equation by some limits.
	Taking the same limits, we expect that the linear relations of integrals become connection coefficients for solutions of Heine's equation.

	\item[(2)] The variant of the $q$-hypergeometric equation $\mathcal{H}_3f(x)=0$ is a special case of the variant of the $q$-Heun equation of degree three.
	Therefore our integrals and series can be considered as the special solutions of the variant of the $q$-Heun equation of degree three.
	The  variant of the $q$-Heun equation of degree three is expressed by the eigenvalue problem for the third degeneration of Ruijsenaars-van Diejen operator $A$ of one variable.
	It is expected that special solutions of the equation $A^{\langle i\rangle}f(x)=Ef(x)$ are obtained by some hypergeometric function.
	Here, $A^{\langle i\rangle}$ is the $i$-th degeneration of $A$.
	In {\cite{Takemura2017}}, it was found that the degenerations of the Ruijsenaars-van Diejen operator correspond to the $q$-Painlev\'{e} equations.
	More precisely, the eigenvalue problem for the degenerations of the Ruijsenaars-van Diejen operator is obtained by a specialization of the linear $q$-difference equation of the Lax pair for the $q$-Painlev\'{e} equation.
	Special solutions for the $q$-Painlev\'{e} equation of type $E_{7}^{(1)}$ are also written by using the functions ${}_{8}W_{7}$.
	In addition, a special solution for the $q$-Painlev\'{e} equation of type $E_{8}^{(1)}$ is written by ${}_{10}W_{9}$.
	{For more details about $q$-hypergeometric solutions of the $q$-Painlev'{e} equations, see \cite{KMNOY}.}
	Therefore, we expect that the eigenvalue problem of the degenerated Ruijsenaars-van Diejen operator has a special solution which is written by ${}_{10}W_{9}$.
	\item[(3)] In Proposition \ref{propconqhyp}, we showed that variants of the $q$-hypergeometric equation are rigid by the configuration.
	In the differential case, a Fuchsian differential equation has an integral solution if the equation is irreducible and rigid.
	In {\cite{Ka}}, it was shown that any irreducible rigid Fuchsian differential equation is reduced to a rank 1 equation by a finite iteration of addition and middle convolution.
	This is called Katz's theory.
	Also in {\cite{DR}}, addition and middle convolution are interpreted as the operations of residue matrices of the Fuchsian equation.
	A $q$-analog of middle convolution was {introduced} in {\cite{SY}}.
	However, a $q$-analog of Katz's theory is not obtained yet.
	We hope that a $q$-analog of Katz's theory is {established}, and many $q$-hypergeometric integrals are explained by this theory.
	Recently, some solutions for the variants of the $q$-hypergeometric equation were also obtained by Arai and Takemura \cite{AT}.
	Their method is based on the $q$-middle convolution.
	It is interesting and important to compare their method with our method.
	Also it seems that our results and their results are examples of a $q$-analog of Katz's theory.
	\item[(4)]In \cite{HMST}, degenerations $\mathcal{H}_{3}f(x)=0$ to $\mathcal{H}_{2}f(x)=0$ and $\mathcal{H}_{2}f(x)=0$ to $\mathcal{H}_{1}f(x)=0$ (which is equivalent to Heine's equation) were studied.
	These are interpreted by configurations (see Remark \ref{Remdegeeq}).
	 By taking the same degenerations for solutions in Theorem \ref{thmint3} and Theorem \ref{thmser3}, we can get several solutions for $\mathcal{H}_{2}f(x)=0$ and $\mathcal{H}_{1}f(x)=0$.
	 This result and detailed calculations will be presented in next paper.
\end{itemize}

\section*{Acknowledgements}
Both authors would like to thank Professor Yasuhiko Yamada for useful discussions and valuable suggestions.
They are also grateful to Professor Wayne Rossman for careful readings and corrections in the manuscript.
This work was supported by JST SPRING, Grant Number JPMJSP2148
and JSPS KAKENHI Grant Number 22H01116.

Address:

Taikei Fujii

Department of Mathematics, 
Graduate School of Science, 
Kobe University, 
Rokko, Kobe 657-8501, Japan.
E-mail: tfujii@math.kobe-u.ac.jp

Takahiko Nobukawa

Department of Education,
Kogakkan University,
 Ise, Mie 516-8555, Japan.
 E-mail: t-nobukawa@kogakkan-u.ac.jp
\end{document}